%% file: ch7.tex
\documentclass[11pt]{article}
\usepackage[utf8]{inputenc}
\usepackage[T1]{fontenc}
\usepackage{microtype}
\usepackage{fix-cm}
\usepackage{amssymb}
\usepackage{amsmath}
\usepackage{graphicx}
\usepackage{subfigure}
\usepackage{multicol}
\usepackage[a4paper]{geometry}
\usepackage[hidelinks]{hyperref}

\usepackage{comment} 
\usepackage{enumerate} 
\usepackage{dsfont} 
\usepackage{mathtools} 
\usepackage{tikz} 
\usepackage{tcolorbox} 
\usepackage{booktabs} 
\usepackage{tabularx} 
\usepackage{array} 

\definecolor{darkteal}{rgb}{0,0.35,0.35}
\definecolor{burntorange}{rgb}{0.8,0.33,0.0}
\newcommand{\js}[1]{\textcolor{darkteal}{\sffamily\small\upshape [JS: #1]}}

\newcommand{\chg}[1]{#1}

\newcommand{\sphere}{\mathbb{S}} 
\newcommand{\diff}{\mathrm{d}} 
\newcommand{\T}{^{\mathrm{\scriptscriptstyle T}}} 
\newcommand{\pr}{\operatorname{\mathsf{P}}} 
\newcommand{\expec}{\operatorname{\mathsf{E}}} 
\newcommand{\rbr}[1]{\left(#1\right)}
\newcommand{\nbr}[1]{\left\|#1\right\|}
\newcommand{\gp}{\operatorname{GP}}
\newcommand{\mgp}{\operatorname{MGP}} 
\newcommand{\gev}{\operatorname{GEV}} 
\newcommand{\LL}{\mathbb{L}} 
\newcommand{\point}{\,\cdot\,} 
\newcommand{\indic}{\operatorname{\mathbb{I}}} 
\newcommand{\evi}{\xi} 
\newcommand{\bevi}{\boldsymbol{\evi}} 

\newcommand{\cbr}[1]{\left\{#1\right\}}

\newcommand{\E}{\mathbb{E}}

\newcommand{\R}{\mathbb{R}}

\def\GP{\text{GP}}

\def\binfty{\boldsymbol \infty}

\def\e{\mathrm{e}}
\newcommand{\Norm}{\operatorname{N}}

\def\bzero{\boldsymbol 0}
\def\bone{\boldsymbol 1}
\def\bA{\boldsymbol A}

\def\bE{\boldsymbol E}

\def\bS{\boldsymbol{S}}

\def\bU{\boldsymbol U}

\def\bX{\boldsymbol X}
\def\bY{\boldsymbol Y}
\def\bZ{\boldsymbol Z}
\def\bM{\boldsymbol M}

\def\bw{\boldsymbol w}

\def\bu{\boldsymbol u}

\def\bv{\boldsymbol v}
\def\bx{\boldsymbol x}
\def\by{\boldsymbol y}
\def\bz{\boldsymbol z}
\def\ba{\boldsymbol a}

\def\bGamma{\boldsymbol \Gamma}
\def\bsigma{\boldsymbol \sigma}
\def\bSigma{\boldsymbol \Sigma}

\def\bT{\boldsymbol T}
\def\bmu{\boldsymbol \mu}
\def\balpha{\boldsymbol \alpha}
\def\bbeta{\boldsymbol \beta}

\newcommand{\vague}{\stackrel{\lower0.2ex\hbox{$\scriptscriptstyle
                    \it{v} $}}{\rightarrow}}
\newcommand{\weak}{\stackrel{\lower0.2ex\hbox{$\scriptscriptstyle
                    \it{w} $}}{\rightarrow}}
\newcommand{\what}{\stackrel{\lower0.2ex\hbox{$\scriptscriptstyle
                    \it{\hat{w}} $}}{\rightarrow}}
\newcommand{\eqdis}{\stackrel{\lower0.2ex\hbox{$\scriptscriptstyle
                    \mathrm{d}$}}{=}}
\newcommand{\distr}{\stackrel{\lower0.2ex\hbox{$\scriptscriptstyle
                    \it{d} $}}{\rightarrow}}
\usepackage{amsthm}
\newtheorem{theorem}{Theorem}[section]
\newtheorem{definition}{Definition}[section]

\newtheorem{proposition}[theorem]{Proposition} 

\theoremstyle{definition}
\newtheorem{example}[theorem]{Example}

\theoremstyle{remark}
\newtheorem{remark}[theorem]{Remark}

\usepackage[
    natbib=true,
    style=numeric,
    backend=biber,
    useprefix=false,
    giveninits=true,
    uniquename=false,
    uniquelist=false,
    maxbibnames=99,
    maxcitenames=1,
]{biblatex}
\begin{document}

\title{Multivariate extreme value theory\footnote{Preliminary version of Chapter~7 of \emph{Handbook on Statistics of Extremes}, edited by Miguel de Carvalho, Raphaël Huser, Philippe Naveau and Brian Reich, to appear at Chapman \& Hall.}}
\author{Philippe Naveau\thanks{LSCE, CNRS/CEA, Gif-sur-Yvette, France. E-mail: philippe.naveau@lsce.ipsl.fr} \and Johan Segers\thanks{Department of Mathematics, KU Leuven, Celestijnlaan 200B 02.28, 3001 Heverlee, Belgium, and LIDAM/ISBA, UCLouvain. E-mail: jjjsegers@kuleuven.be}}
\date{\today}
\maketitle

\begin{abstract}
When passing from the univariate to the multivariate setting, modelling extremes becomes much more intricate. In this introductory exposition, classical multivariate extreme value theory is presented from the point of view of multivariate excesses over high thresholds as modelled by the family of multivariate generalized Pareto distributions. The formulation in terms of failure sets in the sample space intersecting the sample cloud leads to the over-arching perspective of point processes. Max-stable or generalized extreme value distributions are finally obtained as limits of vectors of componentwise maxima by considering the event that a certain region of the sample space does not contain any observation. 
\end{abstract}

\section{Introduction}
\label{ch7:sec:intro}
\input{ch7intro}

\section{Multivariate Generalized Pareto Distributions}
\label{ch7:sec:mgp}
\input{ch7mgp}

\section{Exponent Measures, Point Processes, and More}
\label{ch7:sec:ppp}
\input{ch7ppp}

\section{Multivariate Extreme Value Distributions}
\label{ch7:sec:mev}
\input{ch7mev}

\section{Examples of Parametric Models}
\label{ch7:sec:par}
\input{ch7par}

\section{Notes and Comments}
\label{ch7:sec:lit}
\input{ch7lit}


\section{Mathematical Complements}
\label{ch7:sec:math}
\input{ch7math}

\paragraph{Acknowledgments}

The authors gratefully acknowledge helpful comments by anonymous reviewers on an earlier version of this text.

\printbibliography

\end{document}

%% file: ch7intro.tex
When modelling extremes, the step from one to several variables, even just two, is huge. 
In dimension two or higher, even the very definition of an extreme event is not clear-cut. 
In a multivariate set-up, several questions arise: how to order points in the first place? How to define the maximum of a multivariate sample? 
Similarly, when does a joint observation of several variables exceed a high threshold? Do all coordinates need to be large simultaneously, or just some? Or perhaps it is possible to reduce the multivariate case to the univariate one by applying an appropriate statistical summary such as a projection?
For example, hydrologists often study heavy rainfall by adding precipitation intensities  over  space (regional analysis) or  over time (temporal aggregation).  Given this sum is large, one can wonder how to model extremal dependencies then. 

Despite the variety of questions, it turns out that a common architecture can be built to answer all of them. 
This chapter will highlight how to move from the modeling of multivariate exceedances to a point-process view of extremal event analysis, and finally to connect these two approaches with multivariate block maxima modeling. Historically, the research developments in multivariate extreme value theory have followed a different story-line starting from block maxima, see e.g.\ the chronology between \cite{dehaan:1984} and \cite{Rootzen:Tajvidi:2006}. 
Pedagogically, starting with the multivariate extension of the generalized Pareto distribution appears to be simpler to explain. 

Closely connected to this is the view that a point $\bx = (x_1,\ldots,x_D)\T$ in $D$-dimensional space exceeds a threshold $\bu = (u_1,\ldots,u_D)\T$ as soon as there exists a coordinate $j = 1,\ldots,D$ such that $x_j > u_j$. In words, the point $\bx$ is not dominated entirely by $\bu$, that is, it is \emph{not} true that $\bx \le \bu$, which is written more briefly as $\bx \not\le \bu$. The peaks-over-thresholds (note the double plural) approach that arises from this concept of excess over a high threshold leads to the family of multivariate generalized Pareto distributions.

Multivariate extreme value and multivariate generalized Pareto distributions are two sides of the same coin. They can be understood together by viewing a sample of multivariate observations as a cloud of points in $\R^D$. The vector of componentwise maxima is dominated by a multivariate threshold if and only if no sample point exceeds that threshold if and only if the L-shaped risk region anchored at the vector of thresholds in its elbow does not contain any sample point (Figure~\ref{ch7:fig:max-pot-L}).

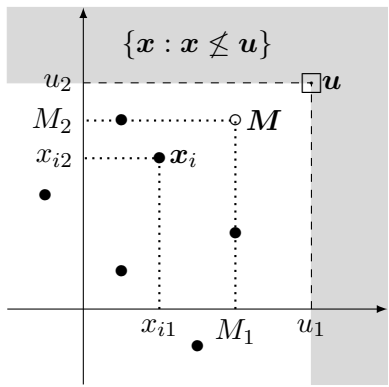
\begin{figure}
\begin{center}
\begin{tikzpicture}[scale=1]
    \filldraw[color=gray!30] (-1,3) -- (3,3) -- (3,-1) -- (4,-1) -- (4,4) -- (-1,4) -- cycle;

    \draw[->,>=latex] (-1,0) -- (4,0);
    \draw[->,>=latex] (0,-1) -- (0,4);

    \draw (3,3) node {$\boxdot$};
    \draw (3,3) node[right] {$\bu$};
    \draw (3,0) node[below] {$u_1$};
    \draw (0,3) node[left] {$u_2$};
    \draw[dashed] (3,0) -- (3,3) -- (0,3);

    \draw (1.5,3.5) node {$\{\bx : \bx \not\le \bu\}$};

    \draw (2,1) node {$\bullet$};
    \draw (-0.5,1.5) node {$\bullet$};
    \draw (1,2) node {$\bullet$};
    \draw (1,2) node[right] {$\bx_i$};
    \draw (1,0) node[below] {$x_{i1}$};
    \draw (0,2) node[left] {$x_{i2}$};
    \draw[dotted,thick] (1,0) -- (1,2) -- (0,2);
    \draw (1.5,-0.5) node {$\bullet$};
    \draw (0.5,2.5) node {$\bullet$};
    \draw (0.5,0.5) node {$\bullet$};

    \draw (2,2.5) node {$\circ$};
    \draw (2,2.5) node[right] {$\bM$};
    \draw (2,0) node[below] {$M_1$};
    \draw (0,2.5) node[left] {$M_2$};
    \draw[dotted,thick] (2,0) -- (2,2.5) -- (0,2.5);
\end{tikzpicture}
\end{center}
\caption{\label{ch7:fig:max-pot-L} The pair $\bM = (M_1,M_2)\T$ of componentwise maxima of a sample of observations in the plane is dominated by a threshold vector $\bu = (u_1,u_2)\T$ if and only if there is no point $\bx_i = (x_{i1},x_{i2})\T$ in the sample that is situated in the L-shaped risk region $\{ \bx : \bx \not\leq \bu \}$ (in gray) defined by $\bu$ in its elbow. Points in the gray region are considered to exceed the multivariate threshold $\bu$.}
\end{figure}

The myriad of possibilities of interactions between two or more random variables requires an entire new set of concepts to model multivariate extremes. Conceptually, it helps to think of the modelling strategy as comprising two parts: first, modelling the univariate marginal distributions of the $D$ variables, and second, modelling their dependence. To make the analogy with the multivariate Gaussian distribution: the univariate distributions of the variables are parameterized in terms of means and variances, while the dependence between the variables is captured by the correlation matrix.

For extremes, the univariate margins can be modelled by parametric families: the univariate generalized extreme value distributions for maxima and the generalized Pareto distributions for excesses over high thresholds. For each variable $j = 1,\ldots,D$, the real-valued shape parameter $\evi_j$ determines how heavy its tail is, and location and scale parameters complete the model.

Alas, for multivariate extremes, no parametric model is able to capture all possible dependence structures. There is no analogue of the correlation matrix to entirely describe all possible interactions between extremes, not even in the classical set-up of multivariate maxima or excesses over high thresholds. \chg{Note moreover that covariances are poorly suited to quantify dependence between variables that are possibly heavy-tailed: first and second moments may not even exist.} Whereas dependence in the Gaussian world can be understood in terms of the classical linear regression model, a full theory of \emph{tail dependence} does not admit such a simple formulation. In the literature, one can find several equivalent mathematical descriptions, some more intuitive than others. In this chapter, we will approach the topic from the angle of multivariate generalized Pareto distributions, using a language that we hope is familiar to non-specialists. 

In the same way as Pearson's linear correlation coefficient describes linear dependence between variables in a way that is unaffected by their location and scale, it is convenient to describe dependence between extremes when each marginal distribution has been individually transformed to the same standardized one. \chg{Removing marginal features allows for sole interpretation of the extremal dependence structure.} In this chapter, we will choose the unit-exponential distribution as pivot, in the same way as the standard normal distribution appears in classical statistics or the uniform distribution on $[0, 1]$ in a copula analysis. The advantage of our choice with respect to other ones in the literature (the unit-Fréchet and unit-Pareto distributions, for instance) is that the formulas are additive rather than multiplicative, thereby resembling, at least superficially, classical models in statistics.
In addition, the loss-of-memory property of the unit exponential distribution facilitates the derivation of properties of the multivariate generalized Pareto distribution, e.g., see Table~\ref{ch7:tab:CheatSheet}.

Multivariate generalized Pareto distributions are introduced in Section~\ref{ch7:sec:mgp}. Viewing high threshold excesses in terms of risk regions intersecting the sample cloud brings us to the over-arching perspective of point processes in Section~\ref{ch7:sec:ppp}. \chg{From there, it is but a small step to the study of multivariate extreme value distributions, by which we mean max-stable distributions, in Section~\ref{ch7:sec:mev}.}
Even though the full theory of tail dependence requires a nonparametric set-up, it is convenient for statistical practice to impose additional structure, for instance in the form of parametric models. These will appear prominently in the later chapters in the book, but we already briefly mention a few common examples in Section~\ref{ch7:sec:par}.
A correct account of the theory requires a bit of formal mathematical language, and, despite our best efforts, some formulas may look less friendly at first sight. We hope the reader will not be put off by these. For those interested, some more advanced arguments are deferred to the end of the chapter in Section~\ref{ch7:sec:math}.

The models developed in this chapter are unsuitable for dealing with situations where the occurrence of extreme values in two or more variables simultaneously is far less frequent than the occurrence of an extreme value in one variable. 
Heavy rainfall from convective storms, for instance, is spatially localized. The probability of large precipitation at two distant locations at the same time is then relatively much smaller than the probability of such an event at one of the two locations. In the literature, this situation is referred to as asymptotic independence, and models developed in this chapter do not possess the correct lenses to measure the relative changes between joint events and univariate ones of this type. More appropriate models that zoom in on this common situation are developed in 
another chapter in the handbook.

\paragraph{Notation.}

For a vector $\bx = (x_1,\ldots,x_D)\T$ and a non-empty set $J \subseteq \{1,\ldots,D\}$, we write $\bx_J = (x_j)_{j \in J}$, a vector of dimension $|J|$, the number of elements in $J$.
Operations between vectors such addition and multiplication are to be understood componentwise. If $\bx = (x_1,\ldots,x_D)\T$ is a vector and $a$ is a scalar, then $\bx - a$ is the vector with components $x_j - a$. The bold symbols $\bone$, $\bzero$ and $\binfty$ refer to vectors all elements of which are equal one, zero, and infinity, respectively. 
Ordering relations between vectors are also meant componentwise: $\bX \le \bu$ means that $X_j \le u_j$ for all components~$j$. The complementary relation is that $\bX \not\le \bu$, which means that $X_1 > u_1$ or \ldots{} or $X_D > u_D$, that is, there exists at least one component $j$ such that $X_j > u_j$. Note that this is different from $\bX > \bu$, which means that $X_j > u_j$ for \emph{all} $j$, that is, $X_1 > u_1$ and \ldots{} and $X_D > u_D$. 
In Figure \ref{ch7:fig:AND.OR}, the \chg{gray} area in the left panel corresponds to  the region  such that $\bX \not\le \bu$  in a bivariate example. The right panel displays the region such that $\bX > \bu$.
%
\begin{figure}
\begin{center}
\begin{tikzpicture}
\draw[color=gray,->,>=latex] (0,0) -- (4.5,0);
\draw[color=gray,->,>=stealth] (0,0) -- (0,4.5);
\filldraw[color=gray!50] (3,3) -- (4,3) -- (4,4) -- (3,4) -- cycle; 
\filldraw[color=gray!50] (0,3) -- (3,3) -- (3,4) -- (0,4) -- cycle; 
\filldraw[color=gray!50] (3,0) -- (4,0) -- (4,3) -- (3,3) -- cycle; 

\draw (2,5) node {$\{\bX \not\le \bu\}$};
\draw (-0.5,3) node {$u_2$};
\draw (3,-0.5) node {$u_1$};
\draw (-0.5+6,3) node {$u_2$};
\draw (2+6,-0.5) node {$u_1$};

\filldraw[color=gray!50] (0+6+2,3) -- (3+6,3) -- (3+6,4) -- (0+6+2,4) -- cycle; 

\draw (2+6,5) node {$\{\bX > \bu\}$};

\draw[color=gray,>=stealth,dashed] (2+6,0) -- (2+6,3);
\draw[color=gray,>=stealth,dashed] (0+6,3) -- (2+6,3);

\draw[color=gray,->,>=stealth] (0+6,0) -- (4+6.5,0);
\draw[color=gray,->,>=stealth] (0+6,0) -- (0+6,4.5);
\end{tikzpicture}

\end{center}
\caption{Extremal regions denoted  by   
    $\{\bX \not\le \bu\}$ (left panel) and $\{\bX > \bu\}$ (right panel) in a schematic bivariate example. 
}
\label{ch7:fig:AND.OR}
\end{figure}
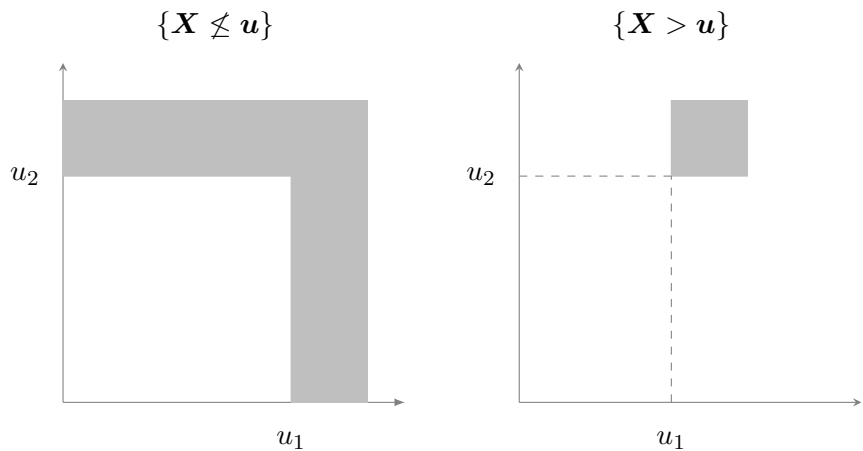


If $\evi = 0$, then $(\e^{\evi z} - 1)/\evi = z$ and $\evi^{-1} \log(1 + \evi z) = z$ by convention.
The indicator variable of an event $A$ is denoted by $\indic_A$ or $\indic(A)$, with value $1$ if $A$ occurs and $0$ otherwise.

%% file: ch7mgp.tex
\paragraph{From univariate to multivariate excesses over high thresholds}

In simple terms, the peaks-over-threshold approach for univariate extremes stipulates that the excess $X-u$ of a random variable $X$ over a high threshold $u$ conditionally on the event $X > u$ can be modelled by the two-parameter family of generalized Pareto distributions. Recall from Chapter~\verb+\ref{ch:whyhow}+ that the conditional distribution of $X - u$ given $X > u$ is approximately $\GP(\sigma_u,\evi)$ with shape parameter $\evi \in \R$ and scale parameter $\sigma_u > 0$, that is,
\[
    \pr(X - u \le x \mid X > u)
    \approx 1 - (1 + \evi x / \sigma_u)_+^{-1/\evi}
    \qquad \text{\chg{uniformly in $x \ge 0$.}}
\]
\chg{The approximation sign is there to indicate that the model only becomes exact in a limiting sense: the difference between the left- and right-hand sides tends to zero as $u$ grows to the upper endpoint of the distribution of $X$, and this uniformly in $x \ge 0$.} In statistical practice, the generalized Pareto distribution is fitted to the observed excesses of a variable over a high threshold. The fitted model is then used as a basis for extrapolation, even beyond the levels observed so far.

As the notation indicates, the scale parameter $\sigma_u$ depends on the threshold $u$. A crucial feature is that, if the high threshold $u$ is replaced by an even higher threshold $v$, the model remains self-consistent: the distribution of excesses over $v$ is again generalized Pareto, with the same shape parameter $\evi$ but a different scale parameter $\sigma_v > u$, which is a function of $\xi$, $\sigma_u$, and $v$.

We would now like to do the same for multivariate extremes. For a random vector $\bX = (X_1,\ldots,X_D)\T$ and a vector of high thresholds $\bu = (u_1,\ldots,u_D)\T$, we seek to model the magnitude of the (multivariate) excess of $\bX$ over $\bu$ conditionally on the event that $\bX$ exceeds $\bu$. But, as already alluded to in the introduction, since $\bX$ and $\bu$ are points in $D$-dimensional space, the meanings of the phrases ``$\bX$ exceeds $\bu$'' and the ``excess of $\bX$ over $\bu$'' are not clear-cut. The most permissive interpretation is to say that for $\bX$ to exceed $\bu$ it is sufficient that there exists at least one $j = 1, \ldots, D$ such that $X_j > u_j$, that is, $\bX \not\le \bu$ (left-hand plot in Figure~\ref{ch7:fig:AND.OR}). Conditionally on this event, the excess is defined as the vector $\bX - \bu = (X_1-u_1,\ldots,X_D-u_D)\T$ of differences, of which at least one is positive, but some others may be negative. As in the univariate case, we seek theoretically justified models for $\bX-\bu$ conditionally on $\bX \not\le \bu$. The threshold vector $\bu$ is taken to be high in the sense that for each $j = 1,\ldots,D$, the probability of the event $X_j > u_j$ is positive but small.

The support of the excess vector $\bX - \bu$ given $\bX \not\le \bu$ requires some closer inspection. As written already, at least one coordinate $X_j - u_j$ must be positive, since there is, by assumption, at least one $j = 1,\ldots,D$ such that $X_j > u_j$. However, the other variables $X_k$ for $k \ne j$ need not exceed their respective threshold $u_k$, and it could thus be that $X_k - u_k \le 0$. The support of the excess vector is therefore included in the somewhat unusual set of points $\bx = (x_1,\ldots,x_D)\T$ with at least one positive coordinate, or formally, such that $\max \bx > 0$. In dimension $D = 2$, this set looks like the letter L written upside down, as the grey area in Figure~\ref{ch7:fig:max-pot-L}. Even for general $D$, we refer to the support of the excess vector as an L-shaped set, even though in dimension $D = 3$, for instance, the support looks more like a large cube from which a smaller cube has been taken out.

\paragraph{Defining multivariate generalized Pareto distributions}

To introduce the family of multivariate generalized Pareto (MGP) distributions, it is convenient to start first on a standardized scale. In dimension one, the generalized Pareto distribution with shape parameter $\evi = 0$ and scale parameter $\sigma = 1$ is just the unit-exponential distribution. If $E$ denotes such a unit-exponential random variable, then for general $\evi \in \R$ and $\sigma > 0$, the distribution of $\sigma (\e^{\evi E} - 1)/\evi$ is $\GP(\sigma,\xi)$. The new ingredient in the multivariate case is the dependence between the $D$ variables, and to focus on this aspect, we first consider  a specific case with standardized margins before we move on to the general case.

There are various ways to introduce and define MGP distributions, see e.g. \cite{Kiriliouk:Rootzen:Segers:2019,Rootzen:Segers:Wadsworth:2018b,Rootzen:Tajvidi:2006}. 
In this section, 
we will construct the MGP family from a common building block:  the unit-exponential distribution.
Other choices could have been made, such as the unit-Fréchet distribution, the Laplace distribution, or also the uniform distribution, for those interested in copulas.
From a pedagogical point of view, we believe that the exponential distribution has many advantages. The exponential seed provides a simple additive representation that permits to define a standard MGP distribution, to generate MGP random samples, to check threshold stability and to deduce properties related to linear combinations and marginalization.

As the reader will notice, the support of the distribution in the next definition includes points with some coordinates equal to minus infinity. This is a theoretical artefact that comes from the chosen scale, and is essentially due to the limits $\log 0 = -\infty$, or, conversely, $\e^{-\infty} = 0$.

\begin{definition}
\label{ch7:def:mgp}
A random vector $\bZ = (Z_1,\ldots,Z_D)\T$ in $[-\infty,\infty)^D$ follows a \emph{standard multivariate generalized Pareto (MGP)} distribution if it satisfies the following two properties:
\begin{enumerate}[(i)]
\item the random variable $E = \max(Z_1,\ldots,Z_D)$ follows a unit-exponential distribution, $\pr(E > z) = \e^{-z}$ for $z \ge 0$;
\item the non-positive random vector 
\begin{equation} 
\label{ch7:eq:Z2S}
    \bS = \bZ - E = (Z_1-E,\ldots,Z_D-E)\T 
\end{equation}
is independent of $E$ and $\pr(S_j > -\infty) > 0$ for all $j = 1,\ldots,D$.
\end{enumerate}
Let $\mgp(\bone, \bzero, \bS)$ denote the distribution of $\bZ$.
\end{definition}

Condition~(i) in Definition~\ref{ch7:def:mgp} implies that with probability one, at least one component of $\bZ$ is positive. This justifies the role of $\bZ$ as a model for the vector $\bX - \bu$ conditionally on $\bX \not\le \bu$, where $\bX$ is a random vector of interest and $\bu$ is a vector of thresholds.
The meaning of the parameters $(\bone, \bzero)$ will become clear in Definition~\ref{ch7:def:mgp:general}.

The support of a standard MGP vector is included in the set
\begin{equation}
\label{ch7:eq:LL}
    \LL = \{ \bx \in [-\infty,\infty)^D : \max \bx > 0 \}. 
\end{equation}
The support of an individual variable $Z_j$ is potentially the whole of $[-\infty,\infty)$, so that $Z_j$ is in general not a unit-exponential random variable. Still, conditionally on $Z_j > 0$, the variable is indeed unit-exponential: as $Z_j = E + S_j$ with $S_j \le 0$ and independent of the unit-exponential random variable $E$, we have, for $x \ge 0$,
\begin{align}
\nonumber
    \pr(Z_j > x) 
    &= \pr(E + S_j > x) \\
    &= \e^{-x} \expec(\e^{S_j}),
    \qquad j = 1, \ldots, D.
\label{ch7:eq:PZjx}
\end{align}
A further special case that often arises from a common marginal standardization is that the probabilities $\pr(Z_j > 0) = \expec(\e^{S_j})$ are equal for all $j = 1,\ldots,D$, but in Definition~\ref{ch7:def:mgp}, this need not be the case.

The parameters $(\bone,\bzero)$ in Definition~\ref{ch7:def:mgp:general} refer to special values of marginal scale and shape parameters, respectively. The general case is as follows.

\begin{definition}
\label{ch7:def:mgp:general}
A random vector $\bY = (Y_1,\ldots,Y_D)\T$ has a \emph{multivariate generalized Pareto} (MGP) distribution if and only if it is of the form
\begin{equation}
\label{ch7:eq:Z2Y}
    \bY = \bsigma \frac{\exp(\bevi \bZ) - 1}{\bevi}
\end{equation}
for scale and shape parameters $\bsigma \in (0,\infty)^D$ and $\bevi \in \R^D$, respectively, where $\bZ \sim \mgp(\bone,\bzero,\bS)$ follows a standard MGP distribution. Let $\mgp(\bsigma, \bevi, \bS)$ denote the distribution of $\bY$. 
\end{definition}

Inverting \eqref{ch7:eq:Z2Y}, we find that a general MGP vector $\bY$ can be reduced to a standard one via
\begin{equation}
\label{ch7:eq:Y2Z}
    \bZ = \frac{1}{\bevi} \log (1 + \bevi \bY / \bsigma),
\end{equation}
where, as usual, all operations are meant componentwise and for $\xi_j = 0$, the formula is meant in a limiting sense, $\lim_{\xi_j \to 0} \evi_j^{-1} \log(1 + \evi_j y_j / \sigma_j) = y_j / \sigma_j$.

For any parameters $\sigma > 0$ and $\evi \in \R$ and for any $z \in \R$, the sign (positive, negative, or zero) of the transformed outcome $\sigma (\e^{\evi z}-1)/\evi$ is the same as that of $z$ itself. For each $j = 1,\ldots,D$, the sign of $Y_j$ in \eqref{ch7:eq:Z2Y} is thus the same as that of $Z_j$. This means that, with probability one, at least one component of $\bY$ is positive, but some components may be negative. If $\evi_j > 0$, the lower bound of $Y_j$ is $-\sigma_j/\evi_j$ rather than $-\infty$.

\begin{remark}
If $\bZ$ follows a standard MGP distribution, the distribution of $\bY = \e^{\bZ}$ is called a \emph{multivariate Pareto distribution}. Its support is included in the set $\{ \bx \in [0, \infty)^d : \max \bx > 1 \}$ and the conditional distribution of $Y_j$ given $Y_j > 1$ is a unit-Pareto.
\end{remark}

\paragraph{MGP distributions as a common-shock model for dependence}

The distribution function\footnote{Or joint cumulative distribution function, to be precise.} of $\bZ \sim \mgp(\bone,\bzero,\bS)$ is determined by the one of $\bS$: from $\bZ = E + \bS$ as in Definition~\ref{ch7:def:mgp}, we get
\begin{equation}
\label{ch7:eq:S2Z}
    \pr(\bZ \le \bx) = \int_0^{\infty} \pr(\bS + z \le \bx) \, \e^{-z} \, \diff z,
    \qquad \bx \in \R^D.
\end{equation}
Conversely, any random vector $\bS = (S_1,\ldots,S_D)\T$ with values in $[-\infty, 0]^D$ satisfying $\max(S_1,\ldots,S_D) = 0$ and $\pr(S_j > -\infty) > 0$ for all $j = 1,\ldots,D$ specifies an MGP distribution function via the right-hand side of Eq.~\eqref{ch7:eq:S2Z}. Hence, specifying an MGP distribution is equivalent to specifying the distribution of a random vector $\bS$ with the two properties in the previous sentence. Such random vectors can be easily constructed by defining
\begin{equation}
\label{ch7:eq:T2S}
    \bS = \bT - \max \bT,
\end{equation}
where $\bT = (T_1,\ldots,T_D)\T$ represents any random vector in $[-\infty,\infty)^D$ such that $\max \bT > -\infty$ almost surely and $\pr(T_j > -\infty) > 0$ for all $j = 1,\ldots,D$. \chg{Choosing a parametric model for $\bT$ is a convenient way to construct one for $\bS$ and thus for $\bZ \sim \mgp(\bone,\bzero,\bS)$.}


The additive structure obtained from Definition~\ref{ch7:def:mgp}, 
\begin{equation}
\label{ch7:eq:Z additive model} 
    \bZ = E+ \bS = E + \bT - \max \bT,
\end{equation}
allows us to comment on  the main features of a standard MGP. 
The common factor, $E$, is the main driver of the system for two reasons. Its  value equally  impacts  all components of $\bZ$ modulo the negative shift produced by $\bS$, and, as $\max \bZ=E$, the largest values of $\bZ$ will always be due to $E$.  
Each component of the non-positive vector $\bS$ indicates how far away the corresponding component of $\bZ$ is from $E$.

For example, if $S_1 = \cdots = S_D = 0$ with probability one, the random vector $\bZ = \mgp(\bone, \bzero, \bS)$ always lies on the diagonal,
\begin{equation}
\label{ch7:eq:compdep}
    \bZ = E \bone = (E,\ldots,E)\T,
\end{equation}
referred to the \emph{complete dependence} case. 
Similarly, if $\min \bS$ is close to zero with large probability then $\bZ$ is close to the point $E \bone$ on the diagonal with large probability, and consequently the dependence structure within $\bZ$ is strong. Likewise, the probability that \emph{all} components of $\bZ$ are positive is 
\[ 
    \pr(\bZ > \bzero) 
    = \pr(Z_1 > 0 \text{ and } \ldots \text{ and } Z_D > 0)
    = \expec(\e^{\min \bS}).
\]
The more concentrated the distribution of $\min \bS$ is around zero, the larger the probability $\pr(\bZ > \bzero)$ becomes.

Because of the common unit-exponential factor $E$ in Eq.~\eqref{ch7:eq:Z additive model}, the components $Z_1,\ldots,Z_D$ of the MGP vector $\bZ$ can never become independent.
Instead, to describe the opposite of complete dependence, consider for $j = 1,\ldots,D$ the event 
\[ 
    A_j = \left\{ S_j = 0 \text{ and } S_k = -\infty \text{ for all } k \ne j \right\}.
\]
The events $A_1,\ldots,A_D$ are mutually exclusive. Now suppose that 
\begin{equation}
\label{ch7:eq:asyindepS}
\left.
\begin{array}{l}
\pr(A_1) + \cdots + \pr(A_D) = 1; \\[1ex]
\pr(A_j) > 0, \qquad j = 1,\ldots,D.
\end{array}
\right\}
\end{equation}
Then the random vector $\bZ \sim \mgp(\bone, \bzero, \bS)$ is of the form 
\[
    \bZ = (-\infty,\ldots,-\infty,E,-\infty,\ldots,-\infty)\T, 
\]
where the unit-exponential random variable $E$ appears at the $j$-th place, with $j$ chosen randomly in $\{1,\ldots,D\}$ with probabilities $\pr(A_j)$. This distribution models the situation where exactly one variable is extreme at a time, a situation referred to as \emph{asymptotic independence}. When translated to multivariate maxima in Section~\ref{ch7:sec:mev}, we will see that it corresponds to independence in the usual sense.



\paragraph{Generating MGP distributions}

By construction, the MGP family corresponds to a non-parametric class as the choice $\bT$ in \eqref{ch7:eq:T2S} is basically free. Still, for most applications, parametric families facilitate interpretation and statistical inference. The simplest way to build a parametric MGP distribution is to impose a parametric form in $\bT$ in Eq.~\eqref{ch7:eq:T2S}, for instance, \chg{a random vector with independent components or a multivariate Gaussian distribution.
In combination with Eq.~\eqref{ch7:eq:Z additive model}, this way of specifying an MGP is particularly convenient for Monte Carlo simulation.}

Another model-building strategy, slightly more complicated but bringing new insights into MGP dependence structures, is based on random vectors of the form
%
\begin{equation}
\label{ch7:eq:XEU}
    \bX = E + \bU,
\end{equation}
where $E$ is a unit-exponential random variable and $\bU = (U_1,\ldots,U_D)\T$ is a random vector in $[-\infty, \infty)^D$, independent of $E$ and such that $0 < \expec(\e^{U_j}) < \infty$ for all $j = 1,\ldots,D$. Since the maximal coordinate of $\bU$ is not necessarily zero, the random vector $\bU$ is in general not a possible vector $\bS$ in Definition~\ref{ch7:def:mgp}, so that $\bX$ in Eq.~\eqref{ch7:eq:XEU} is not necessarily an MGP random vector. Nevertheless, high-threshold excesses of $\bX$ can be shown to be asymptotically MGP distributed in the sense that
\[
    \lim_{u \to \infty} \pr( \bX - u \bone \le \bx \mid \max \bX > u ) = \pr(\bZ \le \bx),
    \qquad \bx \in \R^D,
\]
for $\bZ \sim \mgp(\bone, \bzero, \bS)$, where the distribution of $\bS$ is linked to the one of $\bU$ in the following way: for $\bx \in \R^D$, we have
\begin{equation}
\label{ch7:eq:U2S}
    \pr(\bS \le \bx)
    = \frac{\expec\{\e^Q \indic ( \bU \le \bx + Q ) \}}{\expec[\e^Q]}, \quad
    \text{where } Q = \max(U_1,\ldots,U_D).
\end{equation}

\chg{Eq.~\eqref{ch7:eq:XEU} and~\eqref{ch7:eq:U2S} show how MGP distributions can arise from distributions that are themselves not MGP.}
Eq.~\eqref{ch7:eq:U2S} is mostly of theoretical nature and will be used below to formulate some essential properties of MGP distributions. In practice, we will use other formulas below to compute the $\mgp(\bone, \bzero, \bS)$ probability density from the one of $\bU$. In this way, specific choices of $\bU$ lead to popular parametric models for MGP distributions. Letting $\bU$ have a Gaussian distribution produces the popular Hüsler--Reiss MGP distribution, for instance; see Section~\ref{ch7:sec:par}.

\chg{In our study of the MGP distribution so far, we have introduced several different, but related, random vectors.} The following diagram provides an overview:
\begin{equation}
\label{ch7:eq:TUSZY}
\begin{tikzpicture}[scale=1,baseline=(current bounding box.center)]
    \node (T) at (0, 0.5) {$\bT$};
    \node (U) at (0, -0.5) {$\bU$};
    \node (S) at (2, 0) {$\bS$};
    \node (Z) at (4, 0) {$\bZ$};
    \node (Y) at (6, 0) {$\bY$};

    \draw[->,>=latex] (T) -- (S) node[above,midway,sloped]{\eqref{ch7:eq:T2S}};
    \draw[->,>=latex] (U) -- (S) node[below,midway,sloped]{\eqref{ch7:eq:U2S}};
    \draw[->,>=latex] (S) to[bend left] (Z);
    \draw (3, 0.6) node {\eqref{ch7:eq:S2Z}};
    \draw[->,>=latex] (Z) to[bend left] (S);
    \draw (3,-0.6) node {\eqref{ch7:eq:Z2S}};
    \draw[->,>=latex] (Z) to[bend left] (Y);
    \draw (5, 0.6) node {\eqref{ch7:eq:Z2Y}};
    \draw[->,>=latex] (Y) to[bend left] (Z);
    \draw (5,-0.6) node {\eqref{ch7:eq:Y2Z}};
\end{tikzpicture}
\end{equation}
The numbers above and below the arrows indicate the equations where the corresponding relations are detailed, as we summarize next.

Random vectors $\bT$ and $\bU$ are two different entry points to conveniently generate random vectors $\bS$. We say that a particular distribution is a $\bT$-generator or $\bU$-generator of an MGP distribution if $\bS$ is obtained by letting $\bT$ in Eq.~\eqref{ch7:eq:T2S} or $\bU$ in Eq.~\eqref{ch7:eq:U2S} have that particular distribution. We will see examples of this when presenting some parametric models in Section~\ref{ch7:sec:par}.

The two arrows on the left-hand side of \eqref{ch7:eq:TUSZY} go only one way. This indicates that the distributions of $\bT$ and $\bU$ are not identified by $\bS$. In fact, different choices for the distribution of $\bT$ may lead to the same $\bS$, and similarly for $\bU$. For instance, applying the same common location shift to the components of $\bT$ does not change the position of $\bT - \max \bT$.

The arrows between $\bS$ and $\bZ$ in diagram~\eqref{ch7:eq:TUSZY} signify that the random vector $\bS$ captures the dependence between the $D$ components of $\bZ \sim \mgp(\bone, \bzero, \bS)$ and that the distribution of $\bS$ can in turn be identified from the one of $\bZ$. Finally, passing between the standard case $\bZ$ and the general case $\bY \sim \mgp(\bsigma, \bevi, \bS)$ is just a matter of marginal transformations, and this is the meaning of the arrows on the right-hand side of the diagram.


\paragraph{Tail dependence coefficient}

Suitably chosen dependence coefficients facilitate working with MGP distributions.
To motivate the most common one, let $(Z_1,Z_2)$ be a standard MGP random pair generated by $(S_1,S_2)$ as in Definition~\ref{ch7:def:mgp}. 
Further, let the levels $x_1, x_2 \ge 0$ be such that $\pr(Z_1 > x_1) = \pr(Z_2 > x_2)$. For such $x_1$ and $x_2$, the conditional probability that one of the two variables $Z_j$ exceeds its threshold $x_j$ given that the other variable does so too is
\cite[Proposition~19]{Rootzen:Segers:Wadsworth:2018b}
\begin{align}
\nonumber
    \pr(Z_1 > x_1 \mid Z_2 > x_2)
    &= \pr(Z_2 > x_2 \mid Z_1 > x_1) \\
    &= \expec \left[ \min \left\{ \frac{\e^{S_1}}{\expec(\e^{S_1})}, \frac{\e^{S_2}}{\expec(\e^{S_2})} \right\} \right] =: \chi.
\label{ch7:eq:chi}
\end{align}
The identity is true whatever the values of $x_1,x_2 \ge 0$, as long as the marginal excess probabilities are the same.

The tail dependence coefficient $\chi$ takes values between $0$ and $1$. Since $S_1 = \min(Z_1-Z_2,0)$ and $S_2 = \min(Z_2-Z_1,0)$, the two boundary values of $\chi$ can be interpreted as follows:
\begin{itemize}
\item The case $\chi = 0$ occurs if and only if one of $Z_1$ and $Z_2$ is positive and the other one is $-\infty$. The interpretation is that large values can occur in only one variable at the time---recall that $Z_j$ is a model for the excess $X_j - u_j$ of a variable $X_j$ over a high threshold $u_j$. This case is referred to as \emph{asymptotic independence}.
\item The case $\chi=1$ can only occur if $Z_1 = Z_2$ almost surely. In this case of \emph{complete dependence}, extreme values always appear simultaneously in the two variables, and their magnitudes (after marginal standardization) are the same. 
\end{itemize}
In case of asymptotic independence ($\chi = 0$), the MGP distribution is an uninformative model for describing extremal dependence. In that case, there exists other dependence coefficients and models that are far more adequate. 
We refer to 
later chapters in the handbook
for a detailed coverage.

\paragraph{Stability properties}

Let $\bX$ be a general random vector and $\bu$ a vector of high thresholds. If an MGP distribution serves to model $\bX - \bu$ conditionally on $\bX \not\le \bu$, then for an even higher threshold vector $\bv \ge \bu$, we can compute the distribution of $\bX - \bv$ conditionally on $\bX \not\leq \bv$ in two ways: either directly from $\bX$, or by applying first the MGP model to excesses over $\bu$ and then conditioning these excesses further to exceed the difference $\bv - \bu$. Ideally, both procedures should give the same answer, at least in a limiting sense. A desirable property of MGP distributions is therefore their \emph{threshold stability}, as was explained for their univariate counterparts in the beginning of this section. The following two propositions, derived from Proposition~4 in \cite{Rootzen:Segers:Wadsworth:2018b}, assert that this stability property holds in the multivariate case too, first for the standardized case and subsequently for the general case.

\begin{proposition}[Threshold stability, standard]
\label{ch7:prop:thrstable}
Let $\bZ \sim \mgp(\bone, \bzero, \bS)$ and let $\bu = (u_1,\ldots,u_D)\T \in [0,\infty)^D$. Then 
\[
    \rbr{\bZ - \bu \mid \bZ \nleq \bu} \sim \mgp(\bone, \bzero, \bS_{\bu})
\]
where the distribution of $\bS_{\bu}$ is determined as in Eq.~\eqref{ch7:eq:U2S} with $\bU$ equal to $\bS - \bu$.
If $u_1 = \ldots = u_D$, then $\bS_{\bu}$ has the same distribution as $\bS$, so that, for $u \ge 0$, the distribution of $\bZ - u \bone \mid \bZ \nleq u \bone$ is the same as that of $\bZ$.
\end{proposition}

\begin{proposition}[Threshold stability, general]
\label{ch7:prop:generalstability}
Let $\bY \sim \mgp(\bsigma,\bevi,\bS)$. If $\bv \in [0, \infty)^D$ is such that $\sigma_j + \evi_j v_j > 0$ for all $j$, then
\[
    \rbr{\bY - \bv \mid \bY \nleq \bv} \sim \mgp\rbr{\bsigma + \bevi \bv, \bevi, \bS_{\bu}}
\]
with $u_j = \evi_j^{-1} \log (1 + \evi_j v_j / \sigma_j)$ and for $\bS_u$ as in Proposition~\ref{ch7:prop:thrstable}.
\end{proposition}

The lower-dimensional margins of MGP distributions are not MGP themselves: if $\bY$ is a $D$-variate MGP vector and if $J$ is a proper subset of $\{1,\ldots,D\}$, then the distribution of $\bY_J = (Y_j)_{j \in J}$ is not necessarily MGP. Even a single component is not necessarily a univariate generalized Pareto random variable: we saw this already for the standardized case in the sentences preceding Eq.~\eqref{ch7:eq:PZjx}. The reason is that, even though the whole random vector $\bY$ is guaranteed to have at least one positive component, this positive component need not always occur among the variables in $J$. However, if we condition on the event that the subvector $\bY_J$ has at least one positive component, then we obtain a generalized Pareto distribution again.

\begin{proposition}[Sub-vectors]
\label{ch7:prop:sub}
Let $\bY \sim \mgp(\bsigma, \bevi, \bS)$ and let $J \subseteq \{1,\ldots,D\}$ be non-empty. Then
\[
    \rbr{\bY_J \mid \bY_J \nleq \bzero}
    \sim \mgp(\bsigma_J, \bevi_J, \bS^{(J)})
\]
where the distribution of the $|J|$-dimensional vector $\bS^{(J)}$ is determined as in Eq.~\eqref{ch7:eq:U2S} with $\bU$ equal to $\bS_J$. If the distribution of $\bS$ itself was already determined as in Eq.~\eqref{ch7:eq:U2S} for some random vector $\bU$, then the distribution of $\bS_J$ is generated as in Eq.~\eqref{ch7:eq:U2S} with $\bU$ replaced by $\bU_J$.
\end{proposition}

Notably, the $J$-marginal of the MGP distribution generated by $\bS$ is \emph{not} generated by the $J$-marginal of $\bS$. Some additional transformation, passing by Eq.~\eqref{ch7:eq:U2S} is needed. However, in the latter $\bU$-representation, taking lower-dimensional margins is as simple as taking lower-dimensional margins of $\bU$. This is one of the advantages of the $\bU$-representation.

In case $J$ is a singleton, $J = \{j\}$, Proposition~\ref{ch7:prop:sub} states that $Y_j$ given $Y_j > 0$ is a univariate generalized Pareto random variable. We saw this already in the case $\sigma_j = 1$ and $\xi_j = 0$ in Eq.~\eqref{ch7:eq:PZjx}

The family of MGP distributions satisfies a certain stability property under linear transformations by matrices with positive coefficients, provided the shape parameters $\evi_j$ of all $D$ components are the same. Recall that the components of an MGP vector $\bY$ can be $-\infty$ with positive probability; in a linear transformation as in $\bA \bY$ for an $m \times D$ matrix $\bA$, the convention is that $0 \cdot (-\infty) = 0$. 

\begin{proposition}[Linear transformations]
\label{ch7:prop:SA}
Let $\bY \sim \mgp(\bsigma, \evi \bone, \bS)$ and let $\bA = (a_{i,j})_{i,j} \in [0, \infty)^{m \times D}$ be such that $\pr(\sum_{j=1}^D a_{i,j} Y_j > 0) > 0$ for all $i = 1,\ldots,m$. Then
\[
    \rbr{\bA \bY \mid \bA \bY \nleq \bzero}
    \sim \mgp(\bA \bsigma, \evi \bone, \bS_{\bsigma,\evi,\bA})
\]
where the distribution of $\bS_{\bsigma,\evi,\bA}$ is given by Eq.~\eqref{ch7:eq:U2S} for some random vector $\bU$ whose distribution depends on $\bS,\bsigma,\evi,\bA$.
\end{proposition}

Proposition~\ref{ch7:prop:sub} with $\evi_1 = \ldots = \evi_D$ is a special case of Proposition~\ref{ch7:prop:SA} by an appropriate choice of $\bA$.
A remarkable consequence of Proposition~\ref{ch7:prop:SA} is that, for coefficient vectors $\ba \in [0, \infty)^d$ such that $\pr\rbr{a_1 Y_1 + \cdots + a_D Y_D > 0} > 0$, we have
\[
    \rbr{a_1 Y_1 + \cdots + a_D Y_D \mid a_1 Y_1 + \cdots + a_D Y_D > 0}
    \sim \gp(a_1 \sigma_1 + \cdots + a_D \sigma_D, \evi),
\]
a univariate generalized Pareto distribution whose parameters do not depend on the dependence structure of $\bY$.


\paragraph{Densities}

Calculation of failure probabilities or likelihood-based inference requires formulas for MGP densities. The density of the general case can be easily found in terms of the one in the standard case: the density $p_{\bY}$ of $\bY \sim \mgp(\bsigma, \bevi, \bS)$ in Eq.~\eqref{ch7:eq:Z2Y} can be recovered from the density $p_{\bZ}$ of $\bZ \sim \mgp(\bone, \bzero, \bS)$ by
\begin{equation}
\label{ch7:eq:pZ2pY}
    p_{\bY}(\by) = p_{\bZ}\rbr{\bevi^{-1} \log(1 + \bevi \by / \bsigma)} \prod_{j=1}^D \frac{1}{\sigma_j + \evi_j y_j},
\end{equation}
for $\by \in \R^D$ such that $\by \nleq \bzero$ and $\sigma_j + \evi_j y_j > 0$ for all $j = 1,\ldots,D$.
Here, it is assumed that $\bY$ and $\bZ$ are real-valued, that is, $\pr(S_j = -\infty) = 0$ for all $j = 1,\ldots,D$. 
The extension to the case where $S_j$ can be $-\infty$ with positive probability is explored in \cite{mourahib2024multivariate}.

In view of Eq.~\eqref{ch7:eq:pZ2pY}, it is thus sufficient to study the density of $\bZ \sim \mgp(\bone, \bzero, \bS)$.
Let $\bz \in \R^D$ be such that $\bz \not\leq \bzero$.  Then for $\bS$ generated by $\bT$ as in Eq.~\eqref{ch7:eq:T2S}, the density of $\bZ$ is
\begin{equation}
\label{ch7:eq:mgppdfT}
    p_{\bZ}(\bz) = \frac{1}{\e^{\max \bz}} \int_{-\infty}^{\infty} p_{\bT}(\bz + t) \, \diff t,
\end{equation}
\chg{with $p_{\bT}$ the density of $\bT$.}
In contrast, for $\bS$ generated by $\bU$ as in \eqref{ch7:eq:U2S}, we have
\begin{equation}
\label{ch7:eq:mgppdfU}
    p_{\bZ}(\bz) = \frac{1}{\expec(\e^{\max \bU})} \int_{-\infty}^{\infty} p_{\bU}(\bz + t) \, \e^t \, \diff t,
\end{equation}
\chg{where $p_{\bU}$ is the density of $\bU$.}
For certain distributions of $\bT$ and $\bU$, these integrals can be calculated explicitly, leading to manageable analytic forms for MGP densities; see Section~\ref{ch7:sec:par}. The right-hand sides in Equations~\eqref{ch7:eq:mgppdfT} and~\eqref{ch7:eq:mgppdfU} are similar but different, underlining the different roles of $\bT$ and $\bU$ in the diagram~\eqref{ch7:eq:TUSZY}.

\paragraph{Summary}

In our study of MGP distributions, we have covered a lot of ground already. Table~\ref{ch7:tab:CheatSheet} provides an overview of the various representations and properties. To add some perspective, we have put the MGP distribution in parallel with the multivariate normal distribution.

The last line in Table~\ref{ch7:tab:CheatSheet} deserves some comment. Except for the case of perfect correlation, joint extremes of the multivariate normal distribution feature asymptotic independence: if $\bY \sim \Norm(\bmu, \bSigma)$ and if the correlation between components $Y_i$ and $Y_j$ is not equal to one, then always
\begin{equation}
\label{ch7:eq:YijNorm}
    \lim_{z \to \infty} \frac{\pr \rbr{Y_i > \mu_i + \sigma_i z \text{ and } Y_j > \mu_j + \sigma_j z}}{\pr \rbr{Y_i > \mu_i + \sigma_i z}}
    = 0.
\end{equation}
The probability that $Y_i$ and $Y_j$ both exceed a high critical value is of smaller order than the probability that they do so individually.
In contrast, if $\bY \sim \mgp(\bsigma, \bevi, \bS)$, then, except in the boundary case of asymptotic independence where $\pr(S_i = -\infty \text{ or } S_j = -\infty) = 1$, we have
\begin{equation}
\label{ch7:eq:YijMGP}
    \lim_{z \to \infty}
    \frac{1}{\e^{-z}} \pr \rbr{ Y_i > \sigma_i \frac{\e^{\evi_i z} - 1}{\evi_i} \text{ and } Y_j > \sigma_j \frac{\e^{\evi_j z} - 1}{\evi_j}}
    > 0.
\end{equation}
Joint excesses over high levels thus occur with probabilities that are comparable to those of the corresponding univariate events. The difference between the two situations is fundamental and explains why, to model extremes of multivariate normal random vectors, a different framework is needed, such as the one developed in Chapter~\verb+\ref{ch:cond}+.


\newcolumntype{R}{>{\raggedright\arraybackslash}X}%

\begin{table}[ht]\footnotesize
    \centering
    \begin{tabularx}{\textwidth}{>{\raggedright}p{17ex}|>{\raggedright}p{.3\textwidth}|R}
    \toprule
        & Gaussian $\bY \sim \Norm(\bmu, \bSigma)$ 
        & MGP $\bY \sim \mgp(\bsigma,\bevi,\bS)$ \\
    \midrule\midrule
    \em Parameters 
        & $\bmu$: location
        & $\bsigma$: scale \\ 
        & $\bSigma$: covariance matrix 
        & $\bevi$: shape, extreme value index \\
        &
        & $\bS$: dependence
    \\ \midrule
    \em Definition, generation 
        & $\bY = \bmu + \bGamma \bZ$ with $\bGamma \bGamma\T=\bSigma$ and $\bZ \sim \Norm(\bzero, \boldsymbol{I})$ 
        & $\bY = \bsigma \{\exp(\bevi \bZ) - 1\}/\bevi$ with $\bZ \sim \mgp(\bone, \bzero, \bS)$; further, $\bZ = E + \bS$ with $E \perp \bS$, $E \sim \operatorname{Exp}(1)$ and $\max \bS = 0$, generated by $\bT$ or $\bU$ [diagram~\eqref{ch7:eq:TUSZY}]
    \\ \midrule
    \em Support 
        & $\R^D$ or linear subspace thereof 
        & contained in $\LL = \{ \bx : \bx \nleq \bzero \}$ 
    \\ \midrule
    \em Margins 
        & $\bY_J \sim \Norm(\bmu_{J}, \bSigma_{J})$ 
        & $\rbr{\bY_J \mid \bY_J \nleq \bzero} \sim \mgp(\bsigma_J, \bevi_J, \bS^{(J)}) $ (Proposition~\ref{ch7:prop:sub})
    \\ \midrule
    \em Density (if exists)
        & $p_{\bY}(\by) = |\bSigma|^{-D/2} p_{\bZ}(\bz)$ with $\bZ \sim \Norm(\bzero, \boldsymbol{I})$ and $\bz = \bSigma^{-1/2} (\by - \bmu)$ 
        & $p_{\bY}(\by) = p_{\bZ}(\bz) \prod_{j=1}^d \frac{1}{\sigma_j + \evi_j y_j}$ with $\bz = \frac{1}{\bevi} \log(1 + \bevi \by / \bsigma)$ and $p_{\bZ}(\bz)$ from $p_{\bT}$ or $p_{\bU}$ in \eqref{ch7:eq:mgppdfT}--\eqref{ch7:eq:mgppdfU} 
    \\ \midrule
    \em Stability 
        & sum-stability 
        & threshold-stability (Proposition~\ref{ch7:prop:generalstability})
    \\ \midrule
    \em Linear transformations 
        & $\bA \bY \sim \Norm(\bA \bmu, \bA\bSigma\bA\T)$ for matrix $\bA$ of reals 
        & if $\evi_j = \evi$, then $\rbr{\bA \bY \mid \bA \bY \nleq \bzero} \sim \mgp(\bA \bsigma, \evi \bone, \bS_{\bsigma,\evi,\bA})$ for matrix $\bA$ of nonnegative reals (Proposition~\ref{ch7:prop:SA})
    \\ \midrule
    \em Conditioning
        & $(\bY_2 \mid \bY_1=\by_1) \sim \Norm(\bmu_{2|1}, \bSigma_{2|1})$ 
        & $(\bY - \bv \mid \bY \nleq \bv) \sim \mgp\rbr{\bsigma + \bevi \bv, \bevi, \bS_{\bu}}$ (Proposition~\ref{ch7:prop:generalstability})
    \\ \midrule
    \em Dependence coefficient 
        & linear correlation $\rho$  
        & $\chi=\expec \left[ \min \cbr{ \frac{e^{S_1}}{\expec(e^{S_1})}, \frac{e^{S_2}}{\expec(e^{S_2})}  }\right]$
    \\ \midrule
    \em Tail dependence
        & asymptotic independence \eqref{ch7:eq:YijNorm}
        & asymptotic dependence \eqref{ch7:eq:YijMGP}
    \\ \bottomrule
\end{tabularx}
    \caption{Cheat-sheet comparing properties of Gaussian and MGP distributions.} 
    \label{ch7:tab:CheatSheet}
\end{table}

%% file: ch7ppp.tex
\paragraph{The law of small numbers}

It would have been great if dependence between multivariate extremes could be captured by an object as simple as the correlation matrix of a multivariate normal distribution. As is clear from Section~\ref{ch7:sec:mgp}, things are not that easy. The random vector $\bS$ in Definition~\ref{ch7:def:mgp} describes tail dependence as arising from the individual deviations $S_1,\ldots,S_D$ to a common shock $E$ affecting the whole vector. The additive structure $\bZ = E + \bS$ of the standard MGP distribution can be understood as a random mechanism generating multivariate extremes. However, to understand more advanced models in multivariate extreme value analysis, it is important to grasp another, equivalent object, the \emph{exponent measure}. It is a fundamental notion in multivariate extreme value theory  as it provides the bridge between various concepts and distributions \cite{beirlant:goegebeur:teugels:segers:2004,dehaan:ferreira:2006,resnick:2008,Falk11}.

Suppose you participate in a lottery with a probability of success equal to one in one million ($p = 10^{-6}$), surely a rare event. If you would live long enough to bet at one million different draws (sample size $n = 10^6$), then you could expect to win once ($\lambda = np = 1$), while the number of times you would win the jackpot would be approximately Poisson distributed with parameter $\lambda = 1$: the probability of winning exactly $k$ times, for $k = 0, 1, 2, \ldots$, would be approximately $\e^{-\lambda} \lambda^k / (k!)$. This phenomenon, where a small probability of success is compensated by a large number of trials, is called the \emph{law of small numbers} and underpins much of extreme value theory.

In multivariate extremes, the rare event of interest is not to win the lottery but consists of a risk region of dimension $D$ and may take many shapes. The exponent measure provides the link between the risk region and the Poisson parameter counting the number of points in a large sample that hit the risk region.
The exponent measure associates to multivariate risk regions $B$ a nonnegative number. This number is not a probability nor a density but an \emph{intensity}: it indicates how many points in a large sample can be expected on average to fall in $B$. As the sample size becomes large, there are more candidate observations that can potentially hit $B$. To offset this effect, the set $B$ is pushed away to ever more extreme regions, diminishing the probability of an individual sample point to hit $B$. The two effects are calibrated to counterbalance each other and to reach an equilibrium through the law of small numbers. In the lottery example above, imagine that the number of draws $n$ is further increased but that at the same time, the winning probability $p$ is diminished. As long as the equilibrium $np \to \lambda$ is preserved, the Poisson distribution will emerge eventually.

\paragraph{Exponent measure on unit-exponential scale}

To introduce the exponent measure formally, we first consider the univariate case.
Let $E$ be a unit-exponential variable 
and let $B$ be a subset of the real line with a finite lower bound. For $t$ sufficiently large such that the set $B + t = \{ x + t : x \in B \}$ is contained in $[0, \infty)$, we have, by a change of variables,
\begin{equation}
\label{ch7:eq:urvexp}
    \pr(E \in B + t)
    = \int_{B + t} \e^{-x} \, \diff x
    = \e^{-t} \int_{B} \e^{-x} \, \diff x
    = \e^{-t} \Lambda(B),
\end{equation}
where $\Lambda$ is a measure on $\R$ with density $\lambda(x) = \e^{-x}$ for $x \in \R$: each subset $B$ of $\R$ is mapped to $\Lambda(B) = \int_{B} \e^{-x} \, \diff x$. According to Eq.~\eqref{ch7:eq:urvexp}, the failure probability $\pr(E \in B + t)$ decays as $\e^{-t}$ as $t$ grows, while the proportionality constant is $\Lambda(B)$. The measure $\Lambda$ has two notable properties:
\begin{enumerate}
\item[(i)] it is normalized: $\Lambda([0,\infty)) = 1$;
\item[(ii)] a homogeneity property: $\Lambda(B + t) = \e^{-t} \Lambda(B)$ for $t \in \R$.
\end{enumerate}
Still, the measure $\Lambda$ is \emph{not} a probability measure. In fact, its total mass is infinity, $\Lambda(\R) = +\infty$: indeed, for real $u$, we have $\Lambda([u,\infty)) = \e^{-u}$, and this goes to infinity as $u$ decreases to $-\infty$.

Next, we move to the multivariate case. Let $\bE = (E_1,\ldots,E_D)\T$ 
be a random vector whose $D$ components are all unit-exponential but not necessarily independent.\footnote{The requirement that the margins of $\bE$ are unit-exponential is not essential and we could also assume that $\bE$ is as in Equation~\eqref{ch7:eq:XEU}.} Then, similarly as above, one can investigate the failure probability $\pr(\bE \in B + t)$ as $t$ grows large. It turns out that, in many cases, there exists $0 < \Lambda(B) < \infty$ such that
\begin{equation}
\label{ch7:eq:mrvexp}
    \pr(\bE \in B + t) \sim \e^{-t} \Lambda(B),
    \qquad t \to \infty,
\end{equation}
at least for sets $B \subset [-\infty, \infty)^D$ whose boundary is not too rough and that are bounded from below in our multivariate peaks-over-threshold sense, i.e., there exists $\bu \in \R^D$ such that all $\bx \in B$ satisfy $\bx \not\le \bu$. The symbol $\sim$ in Eq.~\eqref{ch7:eq:mrvexp} means that the ratio of the left and right-hand sides tends to $1$, at least for sets~$B$ such that $0 < \Lambda(B) < \infty$. As in Eq.~\eqref{ch7:eq:urvexp}, the failure probability in \eqref{ch7:eq:mrvexp} decays at rate $\e^{-t}$ and the (asymptotic) proportionality constant depends on $B$ through the factor $\Lambda(B)$. 

Formally, the map $\Lambda$ that associates to set $B$ the proportionality constant $\Lambda(B)$ is a \emph{measure}, i.e., a map that assigns nonnegative numbers to subsets according to certain rules. The measure $\Lambda$ that appears in Eq.~\eqref{ch7:eq:mrvexp} is called an exponent measure for reasons that will become clear in Section~\ref{ch7:sec:mev} when we define multivariate extreme value distributions. In the same way that the individual variables $S_j$ in Definition~\ref{ch7:def:mgp} could hit $-\infty$ with positive probability, the exponent measure $\Lambda$ is defined on the space $[-\infty,\infty)^D \setminus \{-\binfty\}$ of vectors $\bx = (x_1,\ldots,x_D)\T$ with $x_j \in \R \cup \{-\infty\}$ for all $j$ and such that $x_j$ is real-valued (not $-\infty$) for at least one $j$. We say that a subset $B$ of $[-\infty,\infty)^D \setminus \{-\binfty\}$ is bounded away from $-\infty$ if there exists a real $u$ such that all $\bx \in B$ satisfy $\max(x_1,\ldots,x_D) > u$, that is, $\bx$ exceeds the threshold $u \bone$.
As in the univariate case, the exponent measure is normalized in a certain way and is homogeneous.

\begin{definition}
\label{ch7:def:Lam}
Let $\Lambda$ be a measure on $[-\infty,\infty)^D \setminus \{ -\binfty \}$ such that $\Lambda(B)$ is finite whenever $B$ is bounded away from $-\binfty$. Then $\Lambda$ is an \emph{exponent measure} on unit-exponential scale if it satisfies the following two conditions:
\begin{align}
\label{ch7:eq:Lamnorm}
    &\Lambda(\{ \bx : x_j \ge 0 \}) = 1, && j = 1,\ldots,D; \\
\label{ch7:eq:Lamhomo}
    &\Lambda(B + t) = \e^{-t} \Lambda(B), && t \in \R, \; B \subset [-\infty,\infty)^D \setminus \{ -\binfty \}.
\end{align}
\end{definition}

The phrase ``unit-exponential scale'' concerns the identity 
\[ 
    \Lambda(\{ \bx : x_j > u \}) = \e^{-u},
    \qquad j = 1,\ldots,D, \; u \in \R.
\]
Confusingly, perhaps, the name ``exponent measure'' does \emph{not} come from the use of this unit-exponential scale but rather from the appearance of $\Lambda$ in the exponent of the formula for an multivariate extreme value distribution, see Definition~\ref{ch7:def:mev} below.

\paragraph{Point processes}

To see how \chg{the exponent measure $\Lambda$ permits modeling extremes of large samples}, let $\bE_1,\ldots,\bE_n$ be an independent random sample from the distribution of the random vector $\bE$ in Eq.~\eqref{ch7:eq:mrvexp}. Introduce the counting variable
\begin{equation}
\label{ch7:eq:Nn}
    N_n(B) = \sum_{i=1}^n \indic ( \bE_i \in B + \log n ),
\end{equation}
\chg{where $\indic(A)$ denotes the indicator function of the event $A$, equal to $1$ if the event occurs and $0$ otherwise. In \eqref{ch7:eq:Nn}, $N_n(B)$ counts} the number of sample points $\bE_1,\ldots,\bE_n$ in the failure set $B + \log n$. As $n$ grows, there are two opposing effects affecting the distribution of $N_n(B)$: on the one hand, the risk region $B + \log n$ escapes to $+\binfty$, while on the other hand, the number of sample points grows; see Figure~\ref{ch7:fig:Blogn}. The distribution of $N_n(B)$ is Binomial with parameters $n$, the number of ``attempts'', and $\pr(\bE \in B + \log n)$, the probability of ``success''. The translation by $\log n$ is chosen in such a way that the expected number of points in the failure set stabilizes: by Eq.~\eqref{ch7:eq:mrvexp}, we have
\begin{equation}
\label{ch7:eq:ENnLam}
    \expec\cbr{ N_n(B) } = n \pr(\bE \in B + \log n)
    \to \Lambda(B), \qquad n \to \infty.
\end{equation}
By the law of small numbers, the limit distribution of $N_n(B)$ is Poisson with expectation $\Lambda(B)$, provided $0 < \Lambda(B) < \infty$, that is,
\begin{equation}
\label{ch7:eq:toPoisson}
    \lim_{n \to \infty} \pr\cbr{ N_n(B) = k } 
    = \e^{-\Lambda(B)} \frac{\Lambda(B)^k}{k!}, 
    \qquad k = 0, 1, 2, \ldots
\end{equation}
Furthermore, for disjoint sets $B_1$ and $B_2$, the random variables $N_n(B_1)$ and $N_n(B_2)$ become independent as $n \to \infty$.

Together, these facts imply that the point processes $N_n$ converge in distribution to a \emph{Poisson point process} $N$ with \emph{intensity measure} $\Lambda$. The formal theory goes beyond the scope of this chapter. Intuitively, think of $N$ as the joint distribution of a cloud of infinitely many random points $\bX_i$ encoded as a random counting measure $N(B) = \sum_i \indic(\bX_i \in B)$: 
\begin{itemize}
\item
    for each region $B$ such that $\Lambda(B)$ is finite, the number of points $N(B)$ in $B$ is Poisson distributed with parameter $\expec\cbr{N(B)} = \Lambda(B)$;
\item
    for disjoint sets $B_1,\ldots,B_k$, the counting variables $N(B_1),\ldots,N(B_k)$ are independent.
\end{itemize}
While the total number of points in the cloud described by $N$ is an infinite sequence, the number of points that ``exceed'' a threshold $\bu \in \R^D$ (i.e., points $\bX_i$ such that $\bX_i \not\le \bu$) is necessarily finite, that is, $N(B)$ is finite when $B$ remains away from $-\binfty$.

\begin{figure}
\begin{center}
\begin{tikzpicture}
\draw[color=gray,->,>=stealth] (0,0) -- (5,0);
\draw[color=gray,->,>=stealth] (0,0) -- (0,4);
\draw[color=gray] (0,0) node [below left] {$0$};
\filldraw[color=gray!20] (2,2) -- (5,2) -- (5,4) -- (2,4) -- cycle;
\draw[dashed] (1,1) -- (4,1) -- (4,3) -- (1,3) -- cycle;

\draw[->,>=latex,dotted] (1,1) -- (2,2);
\draw[->,>=latex,dotted] (4,1) -- (5,2);
\draw[->,>=latex,dotted] (1,3) -- (2,4);

\draw (1.5,2.5) node {$B$};
\draw (3,3.5) node {$B + \log n$};
\draw (1.2,0.2) node {$\bullet$};
\draw (2.0,0.7) node {$\bullet$};
\draw (2.6,1.5) node {$\bullet$};
\draw (3.2,2.7) node {$\bullet$};
\draw (1.5,1.9) node {$\bullet$};
\draw (1.9,1.3) node {$\bullet$};
\draw (0.7,0.8) node {$\bullet$};
\draw (0.5,0.3) node {$\bullet$};
\draw (3.2,2.7) node[right] {$\bX_i$};
\draw (0.1,1.6) node {$\bullet$};
\draw (1.3,0.7) node {$\bullet$};
\draw (1.3,3.7) node {$\bullet$};
\end{tikzpicture}
\end{center}
\caption{\label{ch7:fig:Blogn} As the sample size $n$ grows, the failure set $B + \log n$ escapes to $+\binfty$. As the sample size $n$ grows, the number $N_n(B)$ of sample points $\bE_1,\ldots,\bE_n$ that fall in the risk region $B + \log n$ converges in distribution to a Poisson random variable with expectation $\Lambda(B)$.}
\end{figure}
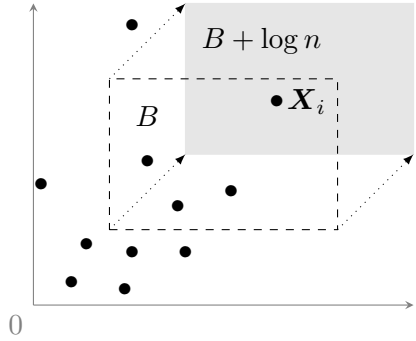

\paragraph{Peaks-over-thresholds}

The exponent measure is connected to the MGP distribution. Recall the set $\LL = \{ \bx : \bx \nleq \bzero \}$ in Eq.~\eqref{ch7:eq:LL} of possible threshold excesses and let $B \subseteq \LL$. In view of Eq.~\eqref{ch7:eq:mrvexp}, we have 
\begin{equation}
\label{ch7:eq:E2LamLL}
    \lim_{t \to \infty} \e^t \pr(\max \bE > t) = \Lambda(\LL)
\end{equation}
and thus
\[
    \lim_{t \to \infty}
    \pr( \bE \in B + t \mid \max \bE > t ) = \frac{\Lambda(B)}{\Lambda(\LL)}
\]
for sufficiently regular\footnote{The topological boundary of $B$ should be a $\Lambda$-null set.} sets $B$. But, for threshold vectors $\bu = t \bone = (t, \ldots, t)\T$, the limit distribution in the previous equation should also be an MGP distribution. We find the following connection; a proof is given in Section~\ref{ch7:sec:math}.

\begin{proposition}
\label{ch7:prop:mgpLam}
If $\Lambda$ is an exponent measure as in Definition~\ref{ch7:def:Lam}, then the distribution of the random vector $\bZ$ defined by
\begin{equation}
\label{ch7:eq:Lam2mgp}
    \pr( \bZ \in B ) = \frac{\Lambda(B)}{\Lambda(\LL)},
    \qquad B \subseteq \LL,
\end{equation}
is standard MGP with 
\begin{equation}
\label{ch7:eq:ESjeq}
    \pr(Z_1 > 0) = \cdots = \pr(Z_D > 0).
\end{equation} 
Conversely, given a standard MGP random vector $\bZ \sim \mgp(\bone,\bzero,\bS)$ that satisfies \eqref{ch7:eq:ESjeq}, we can define an exponent measure $\Lambda$ by
\begin{equation}
\label{ch7:eq:mgp2Lam}
    \Lambda(B) = \frac{1}{\pr(Z_j > 0)} \int_{-\infty}^\infty \pr( t + \bS \in B ) \, \e^{-t} \, \diff t,
\end{equation}
for $B \subset [-\infty,\infty)^D \setminus \{-\binfty\}$, and then Eq.~\eqref{ch7:eq:Lam2mgp} holds. The common value in Eq.~\eqref{ch7:eq:ESjeq} is equal to 
\begin{equation}
\label{ch7:eq:ESjeqlam}
    \pr(Z_j > 0) = \expec(\e^{S_j}) = \frac{1}{\Lambda(\LL)},
    \qquad j = 1,\ldots,D.
\end{equation}
\end{proposition}

The often recurring value $\Lambda(\LL)$ is known as the \emph{extremal coefficient} and lies within the range\footnote{The inequalities follow from $\LL = \bigcup_{j=1}^D \{ \bx : x_j > 0 \}$ and Eq.~\eqref{ch7:eq:Lamnorm}.}
\begin{equation}
\label{ch7:eq:extrcoeff1D}
    1 \le \Lambda(\LL) \le D.
\end{equation}
Its reciprocal can be interpreted as the limiting probability that a specific component of a random vector exceeds a large quantile given that at least one component in that random vector does so: rewriting Eq.~\eqref{ch7:eq:E2LamLL} gives
\[
    \lim_{t \to \infty} \pr(E_j > t \mid \max \bE > t) = \frac{1}{\Lambda(\LL)}.
\]
This interpretation is also in line with Eq.~\eqref{ch7:eq:ESjeqlam}, as we have $\max \bZ > 0$ by definition. In dimension $D = 2$, the extremal coefficient stands in one-to-one relation with the tail dependence coefficient $\chi$ in Eq.~\eqref{ch7:eq:chi} via
\[ 
    \Lambda(\LL) = 2 - \chi.
\]
In general dimension $D$, the larger the extremal coefficient $\Lambda(\LL)$, the smaller the limiting probability $1/\Lambda(\LL)$ and thus the weaker the tail dependence. The two boundary values $\Lambda(\LL) = 1$ and $\Lambda(\LL) = D$ correspond to the cases of complete dependence and asymptotic independence, respectively, as already encountered in the study of MGP distributions:
\begin{itemize}
\item For \emph{complete dependence}, when $\bS = \bzero$ almost surely, $\Lambda$ is concentrated on the diagonal $\{ t \bone : t \in \R \}$, on which it has density $\e^{-t}$; more precisely,
\begin{equation}
\label{ch7:eq:compdepLam}
    \Lambda(B) = \int_{-\infty}^\infty \indic(t \bone \in B) \, \e^{-t} \, \diff t = \int_{t \in \R: t \bone \in B} \e^{-t} \, \diff t.
\end{equation}
\item For the case referred to as \emph{asymptotic independence} in Eq.~\eqref{ch7:eq:asyindepS}, when one randomly chosen component of $\bS$ is zero while all others are $-\infty$, the measure $\Lambda$ is concentrated on the union of the $D$ sets $\{ \bx : x_j \in \R, \max_{k:k\ne j} x_k = -\infty \}$ for $j = 1,\ldots,D$, and on each such set, the density of $\Lambda$ is $\e^{-x_j}$. More precisely, the identity~\eqref{ch7:eq:ESjeqlam} then forces $\pr(S_j = 0) = 1/D$ and thus
\begin{equation}
\label{ch7:eq:asyindepLam}
    \Lambda(B) = \sum_{j=1}^D \int_{-\infty}^{\infty} \indic_B(-\infty,\ldots,x_j,\ldots,-\infty) \, \e^{-x_j} \, \diff x_j,
\end{equation}
where all coordinates of the point in the indicator function are $-\infty$ except for the $j$th one, which is $x_j$.
\end{itemize}

\paragraph{Exponent measure density}

Often, $\Lambda$ does not have any mass on regions where one or more coordinates are $-\infty$ but is concentrated on $\R^D$ or a subset thereof. This happens when $S_j$ is never $-\infty$, for each $j = 1,\ldots,D$. For many models, $\Lambda$ has a density on $\R^D$, denoted by the function $\lambda(\point)$, in the sense that $\Lambda(B) = \int_B \lambda(\bx) \, \diff \bx$ for $B \subseteq \R^D$. By Equations~\eqref{ch7:eq:Lamnorm} and~\eqref{ch7:eq:Lamhomo}, this density then satisfies
\begin{align}
\label{ch7:eq:lamnorm}
    &\int_{\bx \in \R^D: x_j \ge 0} \lambda(\bx) \, \diff \bx = 1, && j = 1,\ldots,D, \\
\label{ch7:eq:lamhomo}
    &\lambda(\bx + t) = \e^{-t} \lambda(\bx), && t \in \R, \; \bx \in \R^D.
\end{align}

The density of the MGP vector $\bZ$ associated to $\Lambda$ as in Eq.~\eqref{ch7:eq:Lam2mgp} is then proportional to $\lambda$:
\begin{equation}
\label{ch7:eq:lam2pZ}
    p_{\bZ}(\bz) = \frac{\lambda(\bz)}{\Lambda(\LL)}, 
    \qquad \bz \nleq \bzero.
\end{equation}
Conversely, by Eq.~\eqref{ch7:eq:ESjeqlam}, the exponent measure density $\lambda$ can be recovered from the probability density of such a random vector $\bZ$ by
\begin{equation}
\label{ch7:eq:pZ2lam}
    \lambda(\bz) = \frac{p_{\bZ}(\bz)}{\pr(Z_j > 0)}, \qquad \bz \nleq \bzero,
\end{equation}
together with the translation property in Eq.~\eqref{ch7:eq:lamhomo}. These formulas allow to pass back and forth between MGP densities and exponent measure densities.


\paragraph{Exponent measure on unit-Pareto scale}

In the literature, the exponent measure is classically defined on $[0, \infty)^D \setminus \{ \bzero \}$ rather than on $[-\infty,\infty)^D \setminus \{-\binfty\}$. If we let the symbol $\nu$ denote this version of the exponent measure, then $\nu$ is connected to $\Lambda$ in Definition~\ref{ch7:def:Lam} via the change of variables from $[-\infty, \infty)$ to $[0, \infty)$ by $x_j \mapsto \e^{x_j}$ for all components $j = 1,\ldots,D$: we have
\begin{equation}
\label{ch7:eq:Lam2nu}
    \nu(B) = \Lambda(\{ \bx : \e^{\bx} \in B \}),
    \qquad B \subseteq [0, \infty)^D \setminus \{ \bzero \}.
\end{equation}
The two conditions on $\Lambda$ in Definition~\ref{ch7:def:Lam} translate into the following requirements on $\nu$:
\begin{align}
\label{ch7:eq:nu1}
    &\nu( \{ \bx : x_j \ge 1 \} ) = 1, && j = 1,\ldots,D; \\
\label{ch7:eq:nuhuomo}
    &\nu(t B) = t^{-1} \nu(B), && 0 < t < \infty, \; B \subseteq [0, \infty)^D \setminus \{ \bzero \}.
\end{align}
These two properties of $\nu$ imply that 
\[ 
    \nu(\{ \bx : x_j > y \}) = 1/y, \qquad y > 0, \; j=1,\ldots,d,
\]
which is why $\nu$ is thought to have ``unit-Pareto scale'', in contrast to the unit-exponential scale of $\Lambda$.

The advantage of using the unit-Pareto scale of $\nu$ rather than the unit-exponential scale of $\Lambda$ is that there is no more need to consider points with some coordinates equal to $-\infty$. When translating things back to multivariate (generalized) Pareto distributions, the drawback is that the additive formulas in Section~\ref{ch7:sec:mgp} become multiplicative. The choice is a matter of taste. Depending on the context, either $\Lambda$ or $\nu$ can be more convenient. To read and understand the extreme value literature, it is helpful to know both and to be aware of their connection. Conceptually, the meaning of $\nu$ is the same as that of $\Lambda$, as a measure of the intensity with which points of a large sample hit an extreme risk region. It is only the univariate marginal scale that is different.

\paragraph{Angular measure}

Another advantage of the unit-Pareto scale exponent measure $\nu$ is that it allows for a geometrical interpretation of dependence between extreme values of different variables. The point of view goes back to the origins of multivariate extreme value theory in \cite{dehaan:resnick:1977}. Given that a multivariate observation exceeds a high threshold, how do the magnitudes of the $D$ variables relate to each other? Imagine the bivariate case. If the point representing the observation lies close to the horizontal axis, it means that the first variable was large, but not the second one. The picture is reversed if the point is situated to the vertical axis. If, however, the point is situated in the vicinity of the diagonal, both variables were large simultaneously, a sign of strong dependence.

To make this more precise, we can use polar coordinates and investigate the distribution of the angular component of those points that exceed a high multivariate threshold. This approach turns out to work best on a Pareto scale. The distribution of angles of extreme observations is called the angular or spectral measure, and many statistical techniques in multivariate extreme value analysis are based on it. 

Recall that a \emph{norm} $\nbr{\point}$ on Euclidean space $\R^D$ is a function that assigns to each point $\bx \in \R^D$ its distance to the origin $\bzero$. The distance can be measured in different ways. The most common norms are the $L_p$-norms for $p \in [1, \infty]$, which are defined by
\[
    \nbr{\bx}_p =
    \begin{dcases}
    \rbr{|x_1|^p + \cdots + |x_D|^p}^{1/p},
    & \text{if $1 \le p < \infty$,} \\
    \max(|x_1|,\ldots,|x_D|),
    & \text{if $p = \infty$.}
    \end{dcases}
\]
The most frequently chosen values are $p = 1, 2, \infty$, yielding the Manhattan (taxi-cab), Euclidean, and Chebyshev or supremum norms, respectively.
The \emph{unit sphere} is the set of points with unit norm, $\{ \bx : \nbr{\bx} = 1 \}$. For the $L_2$-norm, this is the usual sphere in Euclidean geometry, while for $p = \infty$ the unit sphere is actually the surface of the cube $[-1,1]^D$, whereas for $p = 1$ it becomes a diamond or a multivariate analogue thereof.

For a non-zero point $\bx \in \R^D$, consider a generalized version of polar coordinates. The radial component $\nbr{\bx} > 0$ quantifies the overall magnitude of the point. The angular component $\bx / \nbr{\bx}$ is a point on the unit sphere and determines the direction of the point, or more specifically, the half-ray from the origin to the point. In the bivariate case, when $\nbr{\point}$ is the Euclidean norm, we retrieve the traditional polar coordinates $(r, \theta)$ of a point $(x_1,x_2)$ in the plane. The radius is $\nbr{\bx} = \sqrt{x_1^2+x_2^2} = r$ and the angular component is the point $\bx / \nbr{\bx} = (\cos \theta, \sin \theta)$ on the unit circle with angle $\theta \in [0, 2\pi)$.

Recall that the support of the exponent measure $\nu$ on unit-Pareto scale is contained in the positive orthant $[0, \infty)^D$. Thinking of $\nu$ as a kind of distribution, we can imagine the distribution of the angular component $\bx / \nbr{\bx}$ given that the radial component $\nbr{\bx}$ is large. The latter condition can be encoded by $\nbr{\bx} > 1$, because the measure $\nu$ is homogeneous by Eq.~\eqref{ch7:eq:nuhuomo}. The distribution of the angle given that the radius is large is called the \emph{angular measure}. The support of the angular measure is contained in the intersection of $[0, \infty)^D$ with the unit sphere with respect to the chosen norm. This space is denoted here by $\sphere = \{ \bx \in [0, \infty)^D : \nbr{\bx} = 1 \}$ and collects all points in $\R^D$ with nonnegative coordinates and unit norm. 

\begin{definition}
The \emph{angular} or \emph{spectral measure} of an exponent measure $\nu$ with respect to a norm $\nbr{\point}$ on $\R^D$ is the measure $H$ defined on $\sphere$ by
\[
    H(B) 
    = \nu(\{ \bx \ge \bzero : \nbr{\bx} \ge 1, \, \bx / \nbr{\bx} \in B \}),
    \qquad B \subseteq \sphere.
\]
\end{definition}

The homogeneity of $\nu$ implies that it is determined by its angular measure $H$ via
\begin{equation}
\label{ch7:eq:H2nu}
    \nu(\{ \bx \ge \bzero : \nbr{\bx} \ge r, \, \bx / \nbr{\bx} \in B \})
    = r^{-1} H(B)
\end{equation}
for $0 < r < \infty$ and $B \subseteq \sphere$; see Figure~\ref{ch7:fig:ang}. The above formula says that, in ``polar coordinates'' $(r, \bw) = (\nbr{\bx}, \bx / \nbr{\bx}) \in (0, \infty) \times \sphere$, the exponent measure $\nu$ is a product measure with radial component $r^{-2} \, \diff r$ and angular component $H(\diff \bw)$. A measure-theoretic argument beyond the scope of this chapter confirms that $\nu$ can be recovered from $H$.
Modelling $H$ thus provides a way to model $\nu$. An advantage of working with $H$ is that it is supported by the $(D-1)$-dimensional space $\sphere$. Especially for $D = 2$, this simplifies the task of modelling exponent measures to modelling a univariate distribution on a bounded interval.

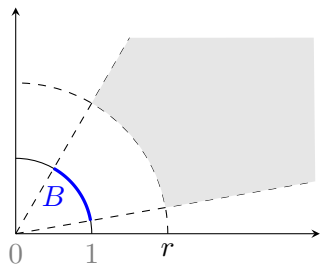
\begin{figure}
\begin{center}
\begin{tikzpicture}[scale=1]
\draw[->,>=stealth] (0,0) -- (0,3);
\draw[->,>=stealth] (0,0) -- (4,0);
\draw[color=gray] (0,0) node[below] {$0$};
\draw[color=gray] (1,0) node[below] {$1$};
\draw (1,0) arc (0:90:1);
\draw[dashed] (2,0) arc (0:90:2);
\draw (2,0) node[below] {$r$};
\draw[domain=10:60,very thick,color=blue] plot (\x:1);
\draw[color=blue] (0.5,0.5) node {$B$};
\draw[dashed] (0, 0) -- (10:4);
\draw[dashed] (0, 0) -- (60:3);
\fill[color=gray!20]
	plot [domain=10:60] (\x:2) --
	(60:3) -- (3.92,2.60) --
	(10:4) -- cycle;
\end{tikzpicture}
\end{center}
\caption{\label{ch7:fig:ang} The exponent measure $\nu$ on $[0, \infty)^D$ is determined by the angular measure $H$ on $\sphere = \{ \bx \ge \bzero : \nbr{\bx} = 1 \}$ by homogeneity via Equation~\eqref{ch7:eq:H2nu}. The set in gray is $\{ \bx \ge \bzero : \nbr{\bx} > r, \, \bx / \nbr{\bx} \in B \}$. In the picture, $\nbr{\point}$ is the Euclidean norm, but other norms can be used too. We think of $H(B)$ in terms of the likelihood of large observations to have a direction in the set $B$, at least when normalized to a certain common scale. According to the direction of the point representing the large observation, one variable may be large without the other being so (directions close to $0^\circ$ or $90^\circ$), or both variables can be large simultaneously (directions close to $45^\circ$).}
\end{figure}

The marginal constraints $\nu(\{ \bx : x_j \ge 1 \}) = 1$ for $j = 1,\ldots,D$ in Eq.~\eqref{ch7:eq:nu1} imply that the angular measure satisfies
\begin{equation}
\label{eq:momentconstraints}
    \int_{\sphere} w_j \, H(\diff \bw) = 1, \qquad j = 1,\ldots,D.
\end{equation}
The total mass of the angular measure $H$ is finite but can vary with $\nu$. A notable exception occurs for the $L_1$-norm, when $\sphere$ becomes the unit simplex $\Delta = \{ \bx \ge \bzero : x_1+\cdots+x_D = 1\}$: adding up the $D$ moment constraints yields a total mass of $H(\Delta) = D$. Dividing $H$ by $D$ then yields a probability distribution, say $P_H(\point) = H(\point)/D$ on the unit simplex, \chg{and this is a matter of preference; in this case, the moment constraints in \eqref{eq:momentconstraints} become $\int_{\Delta} w_j \, P_H(\diff \bw) = 1/D$ for $j = 1,\ldots,D$.}
Models for probability distributions on the unit simplex with all $D$ marginal expectations equal to $1/D$ thus translate directly into models for exponent measures via Eq.~\eqref{ch7:eq:H2nu}. It is in this way, for instance, that the extremal Dirichlet model was constructed in \cite{coles:tawn:1991}.

In case $D = 2$, the unit simplex is equal to the segment $\Delta = \{ (w, 1-w) : w \in [0, 1] \}$, often identified with the interval $[0, 1]$. Modelling bivariate extremal independence is thereby reduced to modelling a probability distribution on $[0, 1]$ with expectation $1/2$. The tail dependence coefficient $\chi$ in Eq.~\eqref{ch7:eq:chi} can be shown to be
\[
    \chi = \int_{[0, 1]} \min(w, 1-w) \, \diff H(w).
\]
The two boundary values of the range of $\chi$ have clear geometric meanings:
\begin{itemize}
\item We have $\chi = 0$ (asymptotic independence) if and only if $H$ is concentrated on the points $(0, 1)$ and $(1, 0)$ of the unit simplex, that is, extreme values can never occur in two variables simultaneously.
\item We have $\chi = 1$ (complete dependence) if and only $H$ is concentrated on the point $(1/2,1/2)$ of the unit simplex, meaning that all extreme points lie on the main diagonal (after marginal standardization).
\end{itemize}

\paragraph{Dependence functions}

If you find all this measure-theoretic machinery a bit heavy, then you are not alone. Computationally, it is often more convenient to work with functions rather than with measures, in the same way that a probability distribution can be identified by its (multivariate) cumulative distribution function. 
The \emph{exponent function} $V$ of an exponent measure $\nu$ or $\Lambda$ is defined by
\begin{equation}
\label{ch7:eq:expfunc}
    V(\by) 
    = \nu(\{ \bx : \bx \nleq \by \}) 
    = \Lambda(\{ \bx : \bx \nleq \log \by \}),
    \qquad \by \in (0, \infty]^D,
\end{equation}
while the \emph{stable tail dependence function} $\ell$ is
\begin{equation}
\label{ch7:eq:stdf}
    \ell(\by) = V(1/\by), \qquad \by \in [0, \infty)^D.
\end{equation}
The exponent function $V$ appears in formula~\eqref{ch7:eq:V2G} below of a multivariate extreme value distribution with unit-Fréchet margins, whereas the stable tail dependence function $\ell$ is convenient when studying extremes from the viewpoint of copulas, a perspective we will not develop in this chapter. The restriction of $\ell$ to the unit simplex $\Delta$ is called the \emph{Pickands dependence function}, and this function determines $\ell$ via the homogeneity in Eq.~\eqref{ch7:eq:Vellhomo} below. The special value
\begin{equation} 
\label{ch7:eq:extrcoeff}
    V(\bone) = \ell(\bone) = \Lambda(\LL) \in [1, D],
\end{equation}
is equal to the extremal coefficient that we already encountered in Eq.~\eqref{ch7:eq:extrcoeff1D}.

The functions $V$ and $\ell$ pop up naturally in the study of multivariate extreme value distributions, see Definition~\ref{ch7:def:mev} and Eq.~\eqref{ch7:eq:V2G} below. Furthermore, the distribution function of the MGP random vector $\bZ$ constructed from $\Lambda$ via \eqref{ch7:eq:Lam2mgp} can be expressed in terms of $V$ and $\ell$ too: rewriting Proposition~4 in \cite{Rootzen:Segers:Wadsworth:2018b}, we find
\begin{equation}
\label{ch7:eq:stdf2Z}
    \pr(\e^{\bZ} \le \by) = \frac{V(\by \wedge \bone) - V(\by)}{V(\bone)},
    \qquad \by \in (0, \infty)^D,
\end{equation}
where $\by \wedge \bone = (\min(y_j, 1))_{j=1}^D$. In particular, we have
\[
    \pr(\e^{\bZ} \nleq \by) = \frac{V(\by)}{V(\bone)}, \qquad \by \ge \bone,
\]
expressing $V$ as a kind of complementary distribution function.

Both functions $V$ and $\ell$ inherit their homogeneity property from $\nu$: by Eq.~\eqref{ch7:eq:nuhuomo}
\begin{equation}
\label{ch7:eq:Vellhomo}
    V(t\by) = t^{-1} V(\by)
    \text{ and }
    \ell(t\by) = t \ell(\by),
    \qquad t \in (0, \infty).
\end{equation}
The marginal constraints on $\nu$ in Eq.~\eqref{ch7:eq:nu1} yield, for all $j = 1, \ldots, D$, the identities
\[
    V(\infty,\ldots,\infty,1,\infty,\ldots,\infty)
    =
    \ell(0,\ldots,0,1,0,\ldots,0)
    =
    1,
\]
where the element $1$ appears on the $j$th place. Furthermore, $\ell$ is convex. Nevertheless, these properties do not characterize the families of exponent functions and stable tail dependence functions. It is for this reason that modelling $V$ or $\ell$ directly is not very practical, as it is difficult to see from their functional form whether they actually constitute a valid exponent measure function or stable tail dependence function, respectively.

The two boundary cases of complete dependence and asymptotic independence permit particularly simple representations in terms of the stable tail dependence function $\ell$. In case of complete dependence, we have\footnote{Formulas~\eqref{ch7:eq:stdfcompdep} and~\eqref{ch7:eq:stdfasyindep} follow for instance from the expressions for $\Lambda$ in Equations~\eqref{ch7:eq:compdepLam} and~\eqref{ch7:eq:asyindepLam} in combination with~Eq.~\eqref{ch7:eq:expfunc}.}
\begin{equation}
\label{ch7:eq:stdfcompdep}
    \ell(\by) = \max(y_1,\ldots,y_D), \qquad \by \in [0, \infty)^D,
\end{equation}
whereas in case of asymptotic independence, we have
\begin{equation}
\label{ch7:eq:stdfasyindep}
    \ell(\by) = y_1 + \cdots + y_D, \qquad \by \in [0, \infty)^D.
\end{equation}
These two expressions will get a straightforward statistical meaning in connection to multivariate extreme value distributions through Eq.~\eqref{ch7:eq:stdf2G} below.

A convenient representation that \chg{permits generation of a valid stable tail dependence function $\ell$} is
\begin{equation}
\label{ch7:eq:Dnorm}
    \ell(\by) = \expec\cbr{\max\rbr{\by \e^{\bU}}},
    \qquad \by \in [0, \infty)^D,
\end{equation}
where $\bU$ is a random vector in $[-\infty,\infty)^D$ such that $\expec[\e^{U_j}] = 1$ for all $j = 1,\ldots,D$. This function $\ell$ is associated to the exponent measure $\Lambda$ obtained as in Proposition~\ref{ch7:prop:mgpLam} from $\bS$ determined in turn by $\bU$ as in Eq.~\eqref{ch7:eq:U2S}. Particular choices of $\bU$ lead to common parametric models for $\Lambda$ and $\ell$, as we will see in Section~\ref{ch7:sec:par}. Formula~\eqref{ch7:eq:Dnorm} identifies $\ell$ as a \emph{D-norm}, an in-depth study of which is undertaken in the monograph \cite{Falk:2019}.


%% file: ch7mev.tex
\paragraph{Multivariate block maxima}

Historically, multivariate extreme value theory begins with the study of distribution functions of vectors of component-wise maxima and asymptotic models for these as the sample size tends to infinity
\cite{beirlant:goegebeur:teugels:segers:2004,dehaan:ferreira:2006,resnick:2008,Falk11}. As we will see, the multivariate extreme value distributions that arise in this way can be understood via the threshold excesses and point processes in the earlier sections.

Recall that the distribution function $F$ of a $D$-variate random vector $\bX$ is the function 
\[ 
	F(\bx) = \pr(\bX \le \bx) = \pr(X_1 \le x_1, \ldots, X_D \le x_D),
	\qquad \bx \in \R^D. 
\]
The (univariate) \emph{margins} $F_1,\ldots,F_D$ of $F$ are the distribution functions of the individual random variables,
\[
	F_j(x_j) = \pr(X_j \le x_j),
	\qquad x_j \in \R, \; j = 1, \ldots, D.
\]
Let $\bX_1,\ldots,\bX_n$ be an independent random sample of $F$, the $i$th sample point being $\bX_i = (X_{i,1},\ldots,X_{i,D})\T$. The sample maximum of the $j$th variable is
\[
    M_{n,j} = \max(X_{1,j},\ldots,X_{n,j})
\]
and joining these maxima into a single vector yields
\begin{equation}
\label{ch7:eq:Mn}
    \bM_n = (M_{n,1},\ldots,M_{n,D})\T.
\end{equation}
The vector $\bM_n$ may not be a sample point itself, since the maxima in the $D$ variables need not occur simultaneously. Still, the study of the distribution of $\bM_n$ has some practical significance: if $X_{i,j}$ denotes the water level on day $i = 1,\ldots,n$ at location $j = 1,\ldots,D$, then, given critical water levels $x_1,\ldots,x_D$ at the $D$ locations, the event $\bM_n \le \bx$ with $n = 365$ signifies that in a given year, no critical level is exceeded at any of the $D$ locations; see Figure~\ref{ch7:fig:max-pot-L} in the case $D = 2$.

Conveniently, the distribution function of $\bM_n$ is related to $F$ in exactly the same way as in the univariate case: for $\bx \in \R^d$, we have $\bM_n \le \bx$ if and only if $\bX_i \le \bx$ for all $i = 1,\ldots,n$. Since the $n$ sample points are independent and identically distributed with common distribution function $F$, we find
\[
    \pr(\bM_n \le \bx) = F^n(\bx).
\]
The aim in classical multivariate extreme value theory is to propose appropriate models for $\bM_n$ based on large-sample calculations, that is, when $n \to \infty$. From the univariate theory in 
an earlier chapter in the handbook,
we know that stabilizing the univariate maxima requires certain location-scale transformations. For each margin $j = 1,\ldots,D$, consider an appropriate scaling $a_{n,j} > 0$ and location shift $b_{n,j} \in \R$, to obtain the location-scale stabilized vector of maxima
\[
    \frac{\bM_n - \boldsymbol{b}_n}{\boldsymbol{a}_n}
    =
    \rbr{ \frac{M_{n,1} - b_{n,1}}{a_{n,1}},\ldots, \frac{M_{n,D} - b_{n,D}}{a_{n,D}}}
\]
with distribution function
\[
    \pr \rbr{\frac{\bM_n - \boldsymbol{b}_n}{\boldsymbol{a}_n} \le \bx}
    = F^n( \boldsymbol{a}_n \bx + \boldsymbol{b}_n ).
\]
As in the univariate case, we wonder what the possible large-sample limits would be. Apart from being of mathematical interest, these limit distributions are natural candidates models for multivariate maxima and can be used \chg{to estimate probabilities such as the ones considered in the previous paragraph.}

\paragraph{Multivariate extreme value distributions}

Recall 
that univariate extreme value distributions form a three-parameter family. The extension to the multivariate case requires specifying how the component variables are related. It is here that the exponent measure $\Lambda$ of Definition~\ref{ch7:def:Lam} comes into play.

\begin{definition}
\label{ch7:def:mev}
A $D$-variate distribution $G$ is a \emph{multivariate extreme value (MEV)} distribution if its margins $G_1,\ldots,G_D$ are univariate extreme value distributions, $G_j = \gev(\mu_j,\sigma_j,\xi_j)$ with $\mu_j \in \R$, $\sigma_j > 0$ and $\xi_j \in \R$ for all $j = 1, \ldots, D$, and there exists an exponent measure $\Lambda$ with 
stable tail dependence function $\ell$ in Eq.~\eqref{ch7:eq:stdf}
such that
\begin{equation}
\label{ch7:eq:stdf2G}
    G(\bx) = \exp [ - \ell\{ -\log G_1(x_1), \ldots, -\log G_D(x_D) \} ]
\end{equation}
for all $\bx \in \R^D$ such that $G_j(x_j) > 0$ for all $j = 1, \ldots, D$.
\end{definition}

In the special case that $G$ has unit-Fréchet margins, $G_j = \gev(1,1,1)$ and thus $G_j(y_j) = \e^{-1/y_j}$ for $y_j > 0$, the expression for $G$ becomes
\begin{equation}
\label{ch7:eq:V2G}
    G(\by) = \e^{-V(\by)} = \e^{-\nu(\{ \bx \ge \bzero : \bx \nleq \by \})} , \qquad \by \in (0, \infty)^D,
\end{equation}
with $V$ the exponent function in Eq.~\eqref{ch7:eq:expfunc} and $\nu$ the exponent measure in Eq.~\eqref{ch7:eq:Lam2nu}. It is Eq.~\eqref{ch7:eq:V2G} from which the exponent function and the exponent measure get their name.\footnote{In the earlier literature, the exponent function is often called exponent measure too.}

We have encountered the two boundary cases of complete dependence and asymptotic independence a number of times in this chapter already. For the stable tail function $\ell$, we found the expressions $\ell(\by) = \max(\by)$ and $\ell(\by) = \sum_{j=1}^D y_j$ in Equations~\eqref{ch7:eq:stdfcompdep} and~\eqref{ch7:eq:stdfasyindep}, respectively. Inserting these into the formula~\eqref{ch7:eq:stdf2G} for the MEV distribution $G$ yields
\[
    G(\bx) =
    \begin{dcases}
        \min \{ G_1(x_1), \ldots, G_d(x_d) \}, & \text{complete dependence,} \\
        G_1(x_1) \cdots G_d(x_d), & \text{asymptotic independence}.
    \end{dcases}
\]
Complete dependence thus corresponds the case where all $D$ variables of which $G$ is the joint distribution are monotone increasing functions of the same random variable.\footnote{Borrowing from copula language, the copula of $G$ is the comonotone one.} Asymptotic independence translates simply to (ordinary) independence. More generally, for $\bx$ such that $G_1(x_1) = \cdots = G_D(x_d) = p \in (0, 1)$, we find
\[
    G(\bx) = p^{\ell(\bone)},
\]
where $\ell(\bone) \in [1, D]$ is the extremal coefficient in Eq.~\eqref{ch7:eq:extrcoeff}. One way to interpret the above formula is that $\ell(\bone)$ is the number of ``effectively independent'' components among the $D$ variables modelled by $G$. In case of complete dependence, we have $\ell(\bone) = 1$, and the $D$ variables behave as a single one. In case of asymptotic independence, we have $\ell(\bone) = D$, as \chg{$G$ factorizes into $D$ independent components.} 

Whereas computing the density of an MGP distribution was a relatively straightforward matter, for MEV distributions things are much more complicated due to the exponential function in Equations~\eqref{ch7:eq:stdf2G} and~\eqref{ch7:eq:V2G} in combination with the chain rule. Successive partial differentiation of $G$ with respect to its $D$ arguments leads to a combinatorial explosion of terms that quickly becomes computationally unmanageable as $D$ grows. This is why, even in moderate dimensions, likelihood-based inference methods for fitting MEV distributions to block maxima are based on other functions than the full likelihood. The issue is especially important for spatial extremes, when the dimension $D$ corresponds to the number of spatial locations.


\paragraph{Max-stability}

For MGP distributions, the characterizing property was threshold stability (Proposition~\ref{ch7:prop:generalstability}). For MEV distributions, the key structural property is \emph{max-stability}. Intuitively, the annual maximum over daily observations is the same as the maximum of the twelve monthly maxima. So if we model both the annual and monthly maxima with an MEV distribution (assuming independence and stationarity of the daily outcomes), we would like the two models for the annual maximum to be mutually compatible. This is exactly what max-stability says. 

\begin{definition}[Max-stability]
A $D$-variate distribution function $G$ is \emph{max-stable} if the distribution of the vector of component-wise maxima of an independent random sample from $G$ is, up to location and scale, equal to $G$. This means that for every integer $k \ge 2$, there exist vectors $\balpha_k \in (0, \infty)^D$ and $\bbeta_k \in \R^D$ such that $G^k(\balpha_k \bx + \bbeta_k) = G(\bx)$ for all $\bx$.
\end{definition}

\begin{proposition}
\label{ch7:prop:maxstable}
A $D$-variate distribution $G$ with non-degenerate margins is an MEV distribution if and only if it is max-stable.
\end{proposition}

\paragraph{Large-sample distributions}

To see where Definition~\ref{ch7:def:mev} comes from and how the exponent measure enters the picture, consider an independent random sample $\bE_1,\ldots,\bE_n$ from the common distribution of a random vector $\bE = (E_1,\ldots,E_D)\T$ with unit-exponential margins. Each point $\bE_i$ is a vector $(E_{i,1},\ldots,E_{i,D})\T$ of $D$ possibly dependent unit-exponent variables. The sample maximum of the $n$ observations of the $j$th variable is now $M_{n,j}^E = \max(E_{1,j},\ldots,E_{n,j})$ and the vector of sample maxima is $\bM_n^E = \rbr{M_{n,1}^E,\ldots,M_{n,D}^E}$.
Assume that the distribution of $\bE$ satisfies Eq.~\eqref{ch7:eq:mrvexp} for some exponent measure $\Lambda$. Recall the counting variable $N_n$ in Eq.~\eqref{ch7:eq:Nn}. As the sample size $n$ grows, the sample maxima diverge to $+\infty$, and in view of Eq.~\eqref{ch7:eq:ENnLam}, the growth rate is $\log n$. The following three statements say exactly the same thing, but using different concepts, namely block maxima, threshold excesses, and point processes: for $\bu \in \R^D$,
\begin{itemize}
\item the vector of sample maxima $\bM_n^E$ is dominated by $\bu$, that is, $\bM_n^E \le \bu$;
\item no point $\bE_i$ exceeds the threshold $\bu$, that is, $\bE_i \le \bu$ for all $i = 1,\ldots,n$;
\item the number of sample points $\bE_1,\ldots,\bE_n$ in $\{ \bz : \bz \nleq \bu \}$ is zero.
\end{itemize}
Now fix $\bx \in \R^D$ and consider the above statements in $\bu = \bx + \log n$. The region $\{ \bz : \bz \nleq \bx + \log n \}$ is of the form $B + \log n$ with $B = \{ \bz : \bz \nleq \bx \}$. By Eq.~\eqref{ch7:eq:toPoisson}, $N_n(B) = \sum_{i=1}^n \indic \rbr{ \bE_i \in B + \log n }$ converges in distribution to a Poisson random variable $N(B)$ with expectation $\expec\cbr{N(B)} = \Lambda(B)$, so that
\begin{align*}
	\pr ( \bM_n^E - \log n \le \bx )
    &= \pr\{ N_n(B) = 0 \} \\
    &\to \pr \{ N(B) = 0 \} = \exp \{ - \Lambda(B) \},
    \qquad n \to \infty.
\end{align*} 
For the given set $B$, we have
\[
	\Lambda(B) = \Lambda(\{ \bz : \bz \nleq \bx \}) = V(\e^{\bx}) = \ell(\e^{-\bx}). 
\]
We obtain
\[
	\lim_{n \to \infty} \pr( \bM_n^E - \log n \le \bx )
	= \exp \{ - \ell(\e^{-\bx}) \} = G(\bx)
\]
with $G$ an MEV distribution as in Definition~\ref{ch7:def:mev} with standard Gumbel margins, $G_j(x_j) = \exp(-\e^{-x_j})$ for $j = 1,\ldots,D$. The same reasoning but for general univariate margins produces $G$ of the form in Eq.~\eqref{ch7:eq:stdf2G}. Recall that a univariate distribution is called non-degenerate if it is not concentrated at a single point but allows for some genuine randomness.

\begin{theorem}[Large-sample distributions of multivariate maxima]
\label{ch7:thm:mda}
Let $\bX_1,\ldots,\bX_n$ be an independent random sample from a common $D$-variate distribution $F$. Assume there exist scaling vectors $\boldsymbol{a}_n \in (0, \infty)^D$ and location vectors $\boldsymbol{b}_n \in \R^D$ together with a multivariate distribution $G$ with non-degenerate margins such that the vector $\bM_n$ in Eq.~\eqref{ch7:eq:Mn} satisfies
\[
    \pr \rbr{\frac{\bM_n - \boldsymbol{b}_n}{\boldsymbol{a}_n} \le \bx}
    = F^n(\boldsymbol{a}_n \bx + \boldsymbol{b}_n) = G(\bx), \qquad n \to \infty,
\]
for all $\bx \in \R^D$ such that $G$ is continuous in $\bx$. Then $G$ is an MEV distribution as in Definition~\ref{ch7:def:mev}.
\end{theorem}

The location-scale sequences $\boldsymbol{a}_n$ and $\boldsymbol{b}_n$ can be found from univariate theory. The new element in Theorem~\ref{ch7:thm:mda} is the \emph{joint} convergence of the $D$ normalized sample maxima. The latter does not follow automatically from the convergence of the univariate maxima separately but is an additional requirement on the relations between the $D$ variables, at least in the tail.

%% file: ch7par.tex
Parametric models for MGP and MEV distributions can be generated by the choice of parametric families for the generator vectors $\bT$ and $\bU$ in Equations~\eqref{ch7:eq:T2S} and~\eqref{ch7:eq:U2S}, respectively. The MGP density then follows by calculating the integrals in Equations~\eqref{ch7:eq:mgppdfT} and~\eqref{ch7:eq:mgppdfU}, while the density of the exponent measure $\Lambda$ follows from Eq.~\eqref{ch7:eq:pZ2lam}. Other ways to construct MGP densities is by exploiting the link in Eq.~\eqref{ch7:eq:lam2pZ} to exponent measure densities and to construct the latter via graphical models \cite{engelke2020graphical} or X-vines \cite{kiriliouk2024x}.


\begin{example}[Logistic family]
\label{ch7:ex:logistic}
In \eqref{ch7:eq:mgppdfU}, the choice of $\bU$ such that each component satisfies $\E\cbr{\exp(U_j)} =  \Gamma(1-1/\alpha)$  for some   $\alpha>1$ leads to  
\[
p_{\bU}(\bz)= \frac{\alpha^{D-1} \Gamma(D-1/\alpha)}{\Gamma(1-1/\alpha)}
\frac{ \exp\cbr{ - \alpha (z_1+\dots +z_D) }}
{\cbr{ \exp\left( - \alpha z_1 \right) + \dots + \exp\left( - \alpha z_D \right)}^{D-1/\alpha}},
\]
with $\bz$ such that $\max(\bz)>0$.
The corresponding  MGP   distribution, $p_{\bZ}(\bz)$ obtained by \eqref{ch7:eq:mgppdfU}, is associated to the well-known logistic max-stable distribution defined with  stable tail dependence function in \eqref{ch7:eq:stdf}
\[
    \ell(\bz) 
    = \rbr{z_1^{1/\alpha} + \dots + z_D^{1/\alpha}}^{\alpha},
    \qquad \bz \geq \bzero.
\]
\end{example}

\begin{example}[Hüsler--Reiss family]
\label{ch7:ex:HR}
A natural choice for $\bU$ in \eqref{ch7:eq:XEU} is a multivariate Gaussian random vector with mean $\bmu$ and 
positive-definite covariance matrix $\bSigma$, i.e.\ $\bU \sim \Norm(\bmu, \bSigma)$.
This gives (see \cite{Kiriliouk:Rootzen:Segers:2019} for details)
$$
p_{\bU}(\bz)= c 
\exp\left[ - \frac{1}{2} 
\left\{ (\bz - \bmu)\T \bA (\bz - \bmu)
+ \frac{2 (\bz - \bmu)\T \bSigma^{-1}\bone - 1}{\bone^T\bSigma^{-1}\bone}
\right\}\right],
$$
with 
$$
c= \frac{(2\pi)^{(1-D)/2}|\bSigma|^{-1/2}}{\expec\cbr{\exp(\max \bU)} (\bone\T \bSigma^{-1} \bone)^{1/2}} 
\mbox{ and }
\bA= \bSigma^{-1} - \frac{\bSigma^{-1}\bone \bone\T \bSigma^{-1}}{\bone\T\bSigma^{-1}\bone},
$$
and for $\bz$ such that $\max(\bz)>0$.
The corresponding  MGP   distribution, $p_{\bZ}(\bz)$ obtained by \eqref{ch7:eq:mgppdfU}, is  associated to the so-called Brown–Resnick or Hüsler--Reiss max-stable model.
The matrix $\bA$ is equal to the Hüsler--Reiss precision matrix studied extensively in \cite{Hentschel09092024}.
This parametric family has been used in various applications, see the chapters on graphical models and max-stable and Pareto processes. 
\end{example}

\begin{example}[$\bT$-Gaussian family]
The previous MGP Hüsler--Reiss distribution should not be confused with the MGP distribution that can be obtained by plugging a multivariate Gaussian random vector as $\bT$ in \eqref{ch7:eq:Z additive model}.  
\chg{In \cite{Kiriliouk:Rootzen:Segers:2019}, the resulting MGP density is derived by calculating the integral in \eqref{ch7:eq:mgppdfT}, yielding}
\[
p_{\bT}(\bz)= 
\frac{(2\pi)^{(1-D)/2}|\bSigma|^{-1/2}}{(\bone^T\bSigma^{-1}\bone)^{1/2}}
\exp\left\{ -\frac{1}{2}(\bz -\bmu)\T \bA (\bz-\bmu) - \max \bz\right\},
\]
for $\bz$ such that $\max(\bz)>0$ and with $\bA$ as above.

\end{example}

Even though the margins of an MEV distribution $G$ belong to the $\gev$ family and thus are continuous, $G$ need not have a $D$-variate density. The most well-known case is the family of max-linear distributions, see e.g. \cite{Kluppelberg22}.


%% file: ch7lit.tex
The goal of this chapter was to provide the mathematical distributional building  blocks to understand 
\chg{multivariate extremes of a $D$-dimensional random sample}. 
In particular, the rich, but complex, connections between the different representations of the dependence structure were highlighted.  
Compared to other book introductions  that open  with  the block-maxima approach to explain concepts in multivariate extreme value theory, our first section  focused on   the MGP distribution. Recent literature 
like \cite{Kiriliouk:Rootzen:Segers:2019,Rootzen:Segers:Wadsworth:2018b} indicates that the MGP family offers new avenues for practitioners  in terms of model building and inference, while its definition remains simple, see \eqref{ch7:def:mgp}. 
Still, the reader needs to keep in mind that all approaches treated in this chapter are inter-connected.
It is often the type of data at hand that  allows the practitioner to select the most appropriate representation in terms of model choice and estimation schemes. 
Inference techniques, simulations algorithms, inclusions  of covariates and various other topics needed to model real life applications will be detailed in the coming chapters, as well as   specific case studies  with dedicated  R code examples. 

Concerning further reading on multivariate extreme value theory,    
a large number of authors have contributed to its development  during the last 30 years, so that a detailed bibliography could be longer than this chapter itself. 
To keep the length of our  list of references at bay, we have arbitrarily decided to  highlight in our short bibliography: general MGP  articles like \cite{Rootzen:Tajvidi:2006,Kiriliouk:Rootzen:Segers:2019}, some  mile-stone stone articles concerning the theoretical  developments  in multivariate extreme value theory, such as \cite{dehaan:resnick:1977},  and methodological or survey papers like \cite{coles:tawn:1991,DavisonHuser15}. Concerning books, readers with mathematical inclinations  could consult 
\cite{dehaan:ferreira:2006,Falk:2019,Falk11,Resnick:2007}.
\chg{An early source is the monograph \cite{resnick:1987}, developing the exponent measure for general max-infinitely divisible distributions in Section~5.3; the measures $\Lambda$ and $\nu$ introduced in our Section~\ref{ch7:sec:ppp} are a special case.}
Some case studies and examples can be found in the book \cite{beirlant:goegebeur:teugels:segers:2004} and of course in the later chapters in this handbook.  
For more recent  references on applications, we simply refer to  the bibliographies  within each chapter of this book. They  offer  another  opportunity to deepen the applied side of the  topics introduced here.


%% file: ch7math.tex

\begin{proof}[Proof of Proposition~\ref{ch7:prop:mgpLam}]
First, suppose $\Lambda$ is an exponent measure as in Definition~\ref{ch7:def:Lam}. We need to show that the random vector $\bZ$ whose distribution is defined in Eq.~\eqref{ch7:eq:Lam2mgp} is a standard mgp random vector as in Definition~\ref{ch7:def:mgp} and that Eq.~\eqref{ch7:eq:ESjeq} holds. The latter follows simply from $\pr(Z_j > 0) = \Lambda(\{\bx : x_j > 0\}) / \Lambda(\LL) = 1/\Lambda(\LL)$, since $\Lambda$ is an exponent measure. To show that $\bZ$ is an mgp random vector, define $E = \max \bZ$. For $t \ge 0$, we have
\[
    \pr(E > t) = \frac{\Lambda(\{\bx : \max \bx > t\})}{\Lambda(\LL)} = \frac{\Lambda(t + \LL)}{\Lambda(\LL)} = \e^{-t},
\]
so that $E$ is a unit-exponential random variable. Further, putting $\bS = \bZ - E$, we have, by homogeneity, for $t \ge 0$ and $A \subset [-\infty, -\infty)^D$,
\begin{align*}
    \pr(E > t, \bS \in A) 
    &= \pr(\max \bZ > t, \bZ - \max \bZ \in A) \\
    &= \frac{\Lambda(\{ \bx : \max \bx > t, \, \bx - \max \bx \in A \})}{\Lambda(\LL)} \\
    &= \frac{\e^{-t} \Lambda(\{ \bx : \max \bx > 0, \, \bx - \max \bx \in A \})}{\Lambda(\LL)} \\
    &= \pr(E > t) \, \pr(\bS \in A),
\end{align*}
yielding the independence of $E$ and $\bS$. The choice $t = 0$ and $A = \{ \bx : x_j > -\infty \}$ yields $\pr(S_j > -\infty) = \Lambda(\{\bx : \max \bx > 0, x_j > -\infty \}) / \Lambda(\LL)$, which is positive, since the numerator is bounded from below by $\Lambda(\{ \bx : x_j > 0 \}) = 1$.

Second, suppose $\bZ \sim \mgp(\bone, \bzero, \bS)$ satisfies Eq.~\eqref{ch7:eq:ESjeq} and define a measure $\Lambda$ by Eq.~\eqref{ch7:eq:mgp2Lam}. Then we need to show that $\Lambda$ is an exponent measure and that Eq.~\eqref{ch7:eq:Lam2mgp} holds. For $j = 1,\ldots,D$, we have
\begin{align*}
    \Lambda(\{\bx : x_j > 0\}) 
    &= \frac{1}{\pr(Z_j > 0)} \int_{-\infty}^\infty \pr(t + S_j > 0) \, \e^{-t} \, \diff t \\
    &= \frac{1}{\pr(Z_j > 0)} \int_{0}^\infty \pr(t + S_j > 0) \, \e^{-t} \, \diff t \\
    &= \frac{1}{\pr(Z_j > 0)} \pr(E + S_j > 0) = 1,
\end{align*}
where $E$ is a unit-exponential random variable independent of $S_j$. On the second line, we used the fact that $S_j \le 0$ and thus $t + S_j \le 0$ for $t \le 0$. Further, for real $u$, the identity $\Lambda(u+B) = \e^{-u} \Lambda(B)$ follows from Eq.~\eqref{ch7:eq:mgp2Lam} by the change of variable from $t$ to $t-u$. Eq.~\eqref{ch7:eq:mgp2Lam} with $B = \LL$ yields
\begin{align*}
    \Lambda(\LL) 
    &= \frac{1}{\pr(Z_j > 0)} \int_{-\infty}^{\infty} \pr(t + \bS \in \LL) \, \e^{-t} \, \diff t \\
    &= \frac{1}{\pr(Z_j > 0)} \int_{0}^{\infty} \e^{-t} \, \diff t \\
    &= \frac{1}{\pr(Z_j > 0)},
\end{align*}
as $\max \bS = 0$ almost surely implies that $\pr(t + \bS \in \LL) = 1$ if $t > 0$ and $\pr(t + \bS \in \LL) = 0$ if $t \le 0$.
Finally, for $B \subseteq \LL$, we have
\begin{align*}
    \Lambda(B) 
    &= \frac{1}{\pr(Z_j > 0)} \int_{-\infty}^\infty \pr(t + \bS \in B) \, \e^{-t} \, \diff t \\
    &= \Lambda(\LL) \int_{0}^\infty \pr(t + \bS \in B) \, \e^{-t} \, \diff t \\
    &= \Lambda(\LL) \pr(E + \bS \in B) = \Lambda(\LL) \pr(\bZ \in B),
\end{align*}
which is Eq.~\eqref{ch7:eq:Lam2mgp}; on the second line, we used the fact that $\pr(t + \bS \in B) = 0$ for $t \le 0$, since $\bS \le \bzero$ and $B \subseteq \LL$.
\end{proof}